\documentclass[review]{elsarticle}

\usepackage{amsfonts,amssymb,amscd,amsmath,enumerate,verbatim,calc,amsthm,mathrsfs,amsxtra,enumerate,verbatim}
\usepackage[all,2cell,ps]{xy}

\usepackage{lineno,hyperref}
\modulolinenumbers[5]

\journal{Journal of \LaTeX\ Templates}

\theoremstyle{plain}

\newtheorem{theorem}{Theorem}[section]
\newtheorem{corollary}[theorem]{Corollary}
\newtheorem{lemma}[theorem]{Lemma}
\newtheorem{proposition}[theorem]{Proposition}

\newtheorem{some new invariant}[theorem]{Some new invariant}
\theoremstyle{definition}

\newtheorem{s}[theorem]{}
\newtheorem{remark}[theorem]{Remark}
\newtheorem{example}[theorem]{Example}

\theoremstyle{remark}
\newtheorem{hypothesis}[theorem]{Hypothesis}

\begin{document}

\begin{frontmatter}

\title{Asymptotic associate primes}

\author{Dipankar Ghosh}
\ead{dghosh@cmi.ac.in}

\address{Chennai Mathematical Institute, H1, SIPCOT IT Park, Siruseri, Kelambakkam, Chennai 603103, Tamil Nadu, India}

\author{Provanjan Mallick}
\ead{prov786@math.iitb.ac.in}

\author{Tony J. Puthenpurakal\corref{mycorrespondingauthor}}
\ead{tputhen@math.iitb.ac.in}

\address{Department of Mathematics, Indian Institute of Technology Bombay, Powai, Mumbai 400076, India}

\cortext[mycorrespondingauthor]{Corresponding author}

\begin{abstract}
We investigate three cases regarding asymptotic associate primes. First, assume $ (A,\mathfrak{m}) $ is an excellent Cohen-Macaulay (CM) non-regular local ring, and $ M = \operatorname{Syz}^A_1(L) $ for some maximal CM $ A $-module $ L $ which is free on the punctured spectrum. Let $ I $ be a normal ideal. In this case, we examine when $ \mathfrak{m} \notin \operatorname{Ass}(M/I^nM) $ for all $ n \gg 0 $. We give sufficient evidence to show that this occurs rarely. Next, assume that $ (A,\mathfrak{m}) $ is excellent Gorenstein non-regular isolated singularity, and $ M $ is a CM $ A $-module with $\operatorname{projdim}_A(M) = \infty $ and $ \dim(M) = \dim(A) -1 $.  Let $ I $ be a normal ideal with analytic spread $ l(I) < \dim(A) $. In this case, we investigate when $\mathfrak{m} \notin \textrm{Ass} \operatorname{Tor}^A_1(M, A/I^n)$ for all $n \gg 0$. We give sufficient evidence to show that this also occurs rarely. Finally, suppose $ A $ is a local complete intersection ring. For finitely generated $ A $-modules $ M $ and $ N $, we show that if $ \operatorname{Tor}_i^A(M, N) \neq 0 $ for some $ i > \dim(A) $, then there exists a non-empty finite subset $ \mathcal{A} $ of $ \operatorname{Spec}(A) $ such that for every $ \mathfrak{p} \in \mathcal{A} $, at least one of the following holds true: (i) $ \mathfrak{p} \in \operatorname{Ass}\left( \operatorname{Tor}_{2i}^A(M, N) \right) $ for all $ i \gg 0 $; (ii) $ \mathfrak{p} \in \operatorname{Ass}\left( \operatorname{Tor}_{2i+1}^A(M, N) \right) $ for all $ i \gg 0 $. We also analyze the asymptotic behaviour of $\operatorname{Tor}^A_i(M, A/I^n)$ for $i,n \gg 0$ in the case when $I$ is principal or $I$ has a principal reduction generated by a regular element.
\end{abstract}

\begin{keyword}
Asymptotic associate primes; Asymptotic grade; Associated graded rings and modules; Local cohomology; Tor; Complete intersections
\MSC[2010] Primary 13A17, 13A30, 13D07; Secondary 13A15, 13H10
\end{keyword}

\end{frontmatter}


\section{Introduction}

In this paper, we investigate three cases regarding asymptotic associate primes. We will introduce it one by one.

\textbf{I:}
Let $(A,\mathfrak{m})$ be a Noetherian local ring of dimension $d$, and let $M$ be a finitely generated $A$-module. By a result of Brodmann \cite{Mb0}, there exists $n_0$ such that the set of associate primes $\operatorname{Ass}_A(M/I^nM)=\operatorname{Ass}_A(M/I^{n_0}M)$ for all $n\geqslant n_0$.  We denote this eventual constant set by $\operatorname{Ass}^{\infty}_I(M)$. 

A natural question is when does $\mathfrak{m} \in \operatorname{Ass}^{\infty}_I(M)$,  or the opposite $\mathfrak{m} \notin \operatorname{Ass}^{\infty}_I(M) $? In general, this question is hopeless to resolve. So well make try to make a few assumptions which are quite general but still amenable to answer our question.

(I.1) We first assume that $(A,\mathfrak{m})$ is an excellent Cohen-Macaulay local ring with infinite residue field. We note that this assumption is quite general.

(I.2) We also assume that $M$ is maximal Cohen-Macaulay (MCM). In fact in the study of modules over Cohen-Macaulay rings the class of MCM modules is the most natural class to investigate. To keep things interesting we also assume $M$ is not free. In particular we are assuming $A$ is also not regular.

We note that in general the  answer to the question on when does $\mathfrak{m} \in \operatorname{Ass}^{\infty}_I(A)$ is \emph{not} known. However, by results of Ratliff, McAdam, a positive answer is known when $I$ is normal, i.e., $I^n$ is integrally closed for all $n \geqslant 1$. In this case, it is known that $\mathfrak{m} \in \operatorname{Ass}^{\infty}_I(A)$ if and only if $ l(I)$, the analytic spread of $I$ is equal to $d = \dim A$; see \cite[4.1]{Mcada06}. So our third assumption is

(I.3) $I$ is a normal ideal of height $ \geqslant 2$.

Before proceeding further, we want to remark that in analytically unramified local rings (i.e., its completion is reduced), there exist plenty of normal ideals. In fact, for any ideal $ I $, it is not terribly difficult to prove that for all $n \gg 0$, the ideal $ \overline{I^n} $ is normal (where $\overline{J}$ denotes the integral closure of an ideal $J$).

Finally, we note that as we are only interested in the question on whether $\mathfrak{m} \in \operatorname{Ass}^{\infty}_I(M)$ or not, it is convenient to assume

(I.4) $M$ is free on the punctured spectrum, i.e., $M_P$ is free for every prime ideal $P \neq \mathfrak{m} $.

We note that (I.4) is automatic if $A$ is an isolated singularity, i.e., $A_P$ is regular for every prime $P \neq \mathfrak{m} $. In general (even when $A$ is not an isolated singularity), any sufficiently high syzygy of a finite length module (of infinite 
projective dimension) will  be free
on the punctured spectrum. 

\begin{remark}
	As discussed above, our hypotheses are satisfied by a large class of rings, modules and ideals. 
\end{remark}

Before stating our results, we need to introduce a few notation. Let $G_I(A) = \bigoplus_{n \geqslant 0}I^n/I^{n+1}$ be the associated graded ring of $A$ with respect to $I$. Let $G_I(A)_+ = \bigoplus_{n \geqslant 1} I^n/I^{n+1}$ be its irrelevant ideal. If $M$ is an $A$-module, then $G_I(M) = \bigoplus_{n \geqslant 0} I^nM/I^{n+1}M$  is the associated graded module of $M$ with respect to $I$
(and considered as an $G_I(A)$-module).

Our first result is

\begin{theorem}\label{first}With assumptions as in {\rm I.1, I.2, I.3 and I.4}, suppose that $d \geqslant 3$, and that $M = \operatorname{Syz}^A_1(L)$ for some MCM $A$-module $L$. We have that
	\[
	\text{if} \ \mathfrak{m} \notin \operatorname{Ass}^{\infty}_I(M), \text{ then } \operatorname{grade}( G_{I^n}(A)_+, G_{I^n}(M)) \geqslant 2  \text{ for all } n \gg 0.
	\]
\end{theorem}

We note that the assumption $M = \operatorname{Syz}^A_1(L)$ for an MCM $ A $-module $L$ is automatically satisfied if $A$ is Gorenstein.

We now describe the significance of Theorem~\ref{first}. The third author of this paper has worked extensively on associated graded rings and modules. He feels that the condition $ \operatorname{grade}( G_{I^n}(A)_+, G_{I^n}(M))  \geqslant 2 $ for all $ n \gg 0 $ is quite special. In `most cases' we will have only $ \operatorname{grade}( G_{I^n}(A)_+, G_{I^n}(M)) = 1 $ for all $ n \gg 0 $. So Theorem~\ref{first} implies that in `most cases' we should have $ \mathfrak{m} \in \operatorname{Ass}^{\infty}_I(M) $.

By a result of Melkersson and Schenzel \cite[Theorem~1]{MS}, it is known that for a finitely generated $A$-module $E$, the set $\operatorname{Ass}_A\left( \operatorname{Tor}^A_1(E, A/I^n) \right)$ is constant for all $ n \gg 0 $. We denote this stable value by $T^\infty_1(I, E)$. We note that if $L$ is a (non-free) MCM $A$-module which is free on the punctured spectrum of $A$, then $\operatorname{Tor}_1^A(L, A/I^n)$ has finite length for all $n \geqslant 1$. In this case, $\mathfrak{m} \notin  T^\infty_1(I, L)$ if and only if $\operatorname{Tor}^A_1(L, A/I^n) = 0$ for all $n \gg 0$. If $\mathfrak{m} \notin \operatorname{Ass}^\infty_I(A)$ (this holds  if $I$ is normal and $l(I) < d$), then it is easy to see that $\operatorname{Tor}^A_1(L, A/I^n) = 0$ for all $n \gg 0$ if and only if $ \mathfrak{m} \notin \operatorname{Ass}^\infty_I(\operatorname{Syz}^A_1(L)) $. If the latter holds, then by Theorem~\ref{first}, we will have $\operatorname{grade}( G_{I^n}( A)_+,  G_{I^n}(\operatorname{Syz}^A_1(L)) \geqslant 2$ for all $ n \gg 0 $. Thus another significance of Theorem \ref{first} is that it suggests that in `most cases' if $I$ is a normal ideal with height $ \geqslant 2$ and $ l(I) < d $, then for a non-free MCM module $ L $, we should have $\operatorname{Tor}^A_1(L, A/I^n)  \neq 0$ for all $n \gg 0$.

\textbf{II:}
In the previous subsection, we consider $T^\infty_1(I, M)$ when $M$ is MCM and locally free on the punctured spectrum. In this subsection, we consider the case when $M$ is Cohen-Macaulay of dimension $d -1$. A trivial case to consider is when $ \operatorname{projdim}_A(M) $ is finite. Then it is easy to see that if $ \mathfrak{m} \notin \operatorname{Ass}^\infty_I(A) $, then $ \mathfrak{m} \notin T^\infty_1(I, M) $.

If projective dimension of $M$ is infinite, then we are unable to analyze $T^\infty_1(I, M)$ for arbitrary Cohen-Macaulay rings. However, we have made progress in this question when $A$ is an isolated Gorenstein singularity.

In general, when a local ring $(R,\mathfrak{n})$ is Gorenstein, and $D$ is a finitely generated $R$-module, there exists an MCM approximation of $D$, i.e., an exact sequence $s \colon  0 \rightarrow Y \rightarrow X \rightarrow D \rightarrow 0$, where $X$ is an MCM $R$-module and $\operatorname{projdim}_R(Y) < \infty$. It is known that if $s^\prime \colon 0 \rightarrow Y^\prime \rightarrow X^\prime \rightarrow D \rightarrow 0$ is another MCM approximation of $D$, then $X$ and $X^\prime$ are stably isomorphic, i.e., there exist finitely generated free $R$-modules $F, G$ such that $X\oplus F \cong X^\prime \oplus G$.  It is clear that $X$ is free if and only if $\operatorname{projdim}_R(D) < \infty$. Thus if $\operatorname{projdim}_R(D) = \infty$, then $\operatorname{Syz}^R_1(X)$ is an invariant of $D$.

Our second result is

\begin{theorem}[= \ref{sir2}]\label{second}
	Let $ (A,\mathfrak{m}) $ be an excellent Gorenstein local ring of dimension $ d \geqslant 3 $. Suppose $ A $ has isolated singularity. Let $I$ be a normal ideal of $A$ with $\operatorname{height}(I)\geqslant2$ and $l(I)<d$. Let $M$ be a Cohen-Macaulay $A$-module of dimension $d-1$ and $\operatorname{projdim}_A(M)=\infty$. Let $s \colon 0 \rightarrow Y \rightarrow X \rightarrow M \rightarrow 0$ be an MCM approximation of $M$. Set $ N := \operatorname{Syz}^A_1(X) $. Then the following statements are equivalent:
	\begin{enumerate}[{\rm (i)}]
		\item $ \mathfrak{m} \notin T^\infty_1(I, M) $ {\rm(}i.e., $ \mathfrak{m} \notin \operatorname{Ass}_A(\operatorname{Tor}^A_1(M,A/{I^n})) $ for all $ n \gg 0 ${\rm )}.
		\item $ \mathfrak{m} \notin \operatorname{Ass}^\infty_I(N) $ {\rm (}equivalently, $ \operatorname{depth}(N/{I^nN}) \geqslant 1 $ for all $ n \gg 0 ${\rm )}.
	\end{enumerate}
	Furthermore, if this holds true, then $ \operatorname{grade}(G_{I^n}(A)_+, G_{I^n}(N)) \geqslant 2 $ for all $ n \gg 0 $.
\end{theorem}

As per our discussion after Theorem~\ref{first}, it follows that in `most cases' we should have $ \mathfrak{m} \in \operatorname{Ass}_A(\operatorname{Tor}^A_1(M,A/{I^n}))~\mbox{for all } n\gg0$.

\textbf{III$\alpha$:}
Let $(A,\mathfrak{m})$ be a local complete intersection ring of codimension $c$. Let $M$ and $N$ be finitely generated $A$-modules. Set
\[
E(M,N) := \bigoplus_{i \geqslant 0}\operatorname{Ext}^i_A(M,N), \quad \mbox{and} \quad T(M,N) := \bigoplus_{i \geqslant 0}\operatorname{Tor}^A_i(M,N).
\]
It is well-known that $ E(M,N)\otimes_A \widehat{A} $ and $ T(M,N)\otimes_A \widehat{A} $ are modules over a ring of cohomology operators $S := \widehat{A}[\xi_1,\ldots, \xi_c]$, where $ \widehat{A} $ is the $ \mathfrak{m} $-adic completion of $ A $. Moreover, $E(M,N)\otimes_A \widehat{A}$ is a finitely generated graded $S$-module. But the $S$-module $T(M,N)\otimes_A \widehat{A}$ is very rarely finitely generated. However, by a result of Gulliksen \cite[Theorem~3.1]{G}, if $\operatorname{Tor}_i^A(M,N)$ has finite length for all $i \gg 0$ (say from $i \geqslant i_0$), then the $ S $-submodule
\[
	T_{\geqslant i_0}(M,N) \otimes_A \widehat{A} \;\; := \; \bigoplus_{i \geqslant i_0}\operatorname{Tor}^{\widehat{A}}_i\big(\widehat{M},\widehat{N}\big) \quad \mbox{is *Artinian.}
\]

By standard arguments, for each $ l = 0, 1 $, it follows that $ \operatorname{Ass}_A(\operatorname{Ext}^{2i+l}_A(M,N)) $ is a constant set for all $ i \gg 0 $. However, we do not have a similar result for Tor. By a result of Avramov and Buchweitz (Theorem~\ref{theorem: vanishing of Tor}), the case when $\operatorname{Tor}^A_i(M, N) = 0$ for all $i \gg 0$ is well-understood. Our third result is that

\begin{theorem}[$ = $ \ref{corollary: asymptotic Ass on Tor}]\label{third-alpha}
	Let $ A $ be a local complete intersection ring. Let $ M $ and $ N $ be finitely generated $ A $-modules.  Assume that $ \operatorname{Tor}_i^A(M, N) \neq 0 $ for some $ i > \dim(A) $. Then there exists a non-empty finite subset $ \mathcal{A} $ of $ \operatorname{Spec}(A) $ such that for every $ \mathfrak{p} \in \mathcal{A} $, at least one of the following statements holds true:
	\begin{enumerate}[{\rm (i)}]
		\item $ \mathfrak{p} \in \operatorname{Ass}_A\left( \operatorname{Tor}_{2i}^A(M, N) \right) $ for all $ i \gg 0 $;
		\item $ \mathfrak{p} \in \operatorname{Ass}_A\left( \operatorname{Tor}_{2i+1}^A(M, N) \right) $ for all $ i \gg 0 $.
	\end{enumerate}
\end{theorem}

\textbf{III$\beta$:}
In \cite[Corollary~4.3]{GP}, the first and the third author proved that if $(A,\mathfrak{m})$ is a local complete intersection ring, $I$ is an ideal of $A$, and $M, N$ are finitely generated $A$-modules, then for every $l = 0, 1$, the set $\operatorname{Ass}_A(\operatorname{Ext}_A^{2i+l}(M, N/I^nN))$ is constant for all $i, n \gg 0$. We do not have a similar result for Tor. It follows from \cite[Theorem~6.1]{GP} that complexity of $N/I^n N$ is stable for all $ n \gg 0 $. Thus, by results of Avramov and Buchweitz, the case when $\operatorname{Tor}^A_i(M, N/I^nN) = 0$ for all $ i, n \gg 0 $ is well-understood. Our final result is

\begin{theorem}[$ = $ \ref{corollary: asymptotic ass: Tor: for special ideals}]\label{third-beta}
	Let $ A $ be a local complete intersection ring. Let $ M $ be a finitely generated $ A $-module, and $ I $ be an ideal of $ A $. Suppose either $ I $ is principal or $ I $ has a principal reduction generated by an $ A $-regular element. Then there exist $ i_0 $ and $ n_0 $ such that either $ \operatorname{Tor}^A_i(M, N/I^nN) = 0 $ for all $ i \geqslant i_0 $ and $ n \geqslant n_0 $, or there is a non-empty finite subset $ \mathcal{A} $ of $ \operatorname{Spec}(A) $ such that for every $ \mathfrak{p} \in \mathcal{A} $, at least one of the following statements holds true:
	\begin{enumerate}[{\rm (i)}]
		\item $ \mathfrak{p} \in \operatorname{Ass}_A\left( \operatorname{Tor}_{2i}^A(M, A/I^n) \right) $ for all $ i \geqslant i_0 $ and $ n \geqslant n_0 $;
		\item $ \mathfrak{p} \in \operatorname{Ass}_A\left( \operatorname{Tor}_{2i+1}^A(M, A/I^n) \right) $ for all $ i \geqslant i_0 $ and $ n \geqslant n_0 $.
	\end{enumerate}
\end{theorem}

\emph{Techniques used to prove our results:}
To prove Theorems~\ref{third-alpha} and \ref{third-beta}, we use the well-known technique of Eisenbud operators over resolutions of modules over complete complete-intersection rings.  We also use results of Gulliksen and Avramov-Buchweitz stated above.       	

However, to  prove Theorems~\ref{first} and \ref{second}, we use a new technique in the study of asymptotic primes, i.e., we investigate the function
\[
\xi_M^I(n) := \operatorname{grade}(G_{I^n}(A)_+, G_{I^n}(M)).
\] 
Note that by a result of Elias \cite[Proposition~2.2]{E}, we have that $\operatorname{depth}(G_{I^n}(A))$ is constant for all $n \gg 0$ (and a similar argument works for modules). However, the function $\xi_M^I$ (when $\dim(A/I) > 0$) has not been investigated before, neither in the study of blow-up algebra's or with the connection with associate primes. Regarding $\xi_M^I$, we prove two results: The first is

\begin{theorem}[$ = $ \ref{thm-xi-1}]\label{xi-1}
	Let $ (A,\mathfrak{m}) $ be a Noetherian local ring. Let $ I $ be an ideal of $ A $, and $M$ be a finitely generated $ A $-module such that $ \operatorname{grade}(I,M) = g > 0 $. Then either $ \operatorname{grade}(G_{I^n}(A)_+, G_{I^n}(M)) = 1 $ for all $ n \gg 0 $, or
	\[
		\operatorname{grade}(G_{I^n}(A)_+, G_{I^n}(M)) \geqslant 2  \quad \mbox{for all } \; n \gg 0.
	\] 
\end{theorem}

Our second result in this direction is

\begin{theorem}[$ = $ \ref{thm-xi-2}]\label{xi-2}
	Let $ (A,\mathfrak{m}) $ be a Cohen-Macaulay local ring, and $ I $ be an ideal of $ A $ such that $ \operatorname{height}(I) \geqslant \dim(A) - 2 $. Let $ M $ be an MCM $ A $-module. Then $ \operatorname{grade}(G_{I^n}(A)_+, G_{I^n}(M)) $ is constant  for all $ n \gg 0 $.
\end{theorem}

Although he does not have an example, the third author feels that $ \xi_M^I $ may not be constant for $ n \gg 0 $ if $ \dim(A/I) \geqslant 3 $.

\begin{remark}
	In \cite[Theorem~3.4]{Sh99}, Huckaba and Marley proved that if $ A $ is Cohen-Macaulay of dimension $ \geqslant 2 $, and if $ I $ is a normal ideal with $ \operatorname{grade}(I) \geqslant 1 $, then $ \operatorname{depth}(G_{I^n}(A)) \geqslant 2 $ for all $ n \gg 0 $.  A crucial ingredient for the proofs of our results is to compute $\operatorname{grade}(G_{I^n}(A)_+,G_{I^n}(A))$ for all $n\gg0$ when $I$ is normal. So we prove the following result.
\end{remark}

\begin{theorem}[$ = $ \ref{RR}]\label{normal}
	Let $ A $ be an excellent Cohen-Macaulay local ring. Let $ I $ be a normal ideal of $A$ such that $ \operatorname{grade}(I) \geqslant 2 $. Then $ \operatorname{grade}(G_{I^n}(A)_+,G_{I^n}(A)) \geqslant 2 $ for all $ n \gg 0 $.
\end{theorem}

We now describe in brief the contents of this paper. In Section~\ref{sec2}, we discuss a few preliminaries on grade and local cohomology that we need.  In Section~\ref{sec3}, we investigate the function $\xi_M^I(n) := \operatorname{grade}(G_{I^n}(A)_+, G_{I^n}(M))$ and prove Theorems~\ref{xi-1}, \ref{xi-2} and \ref{normal}. In Section~\ref{sec4}, we prove Theorems~\ref{first} and \ref{second}.  Finally, in Section~\ref{sec5}, we prove Theorems~\ref{third-alpha} and \ref{third-beta}.

\section{Preliminaries on grade and local cohomology}\label{sec2}

Throughout this article, all rings are assumed to be commutative Noetherian rings with identity. Throughout, let $(A,\mathfrak{m})$ be a local ring of dimension $d$ with infinite residue field, and $M$ be a finitely generated $A$-module. Let $I$ be an ideal of $A$ (which need not be $\mathfrak{m} $-primary). If $p \in M$ is non-zero, and $j$ is the largest integer such that $p\in {I}^{j} M$, then $p^{*}$ denotes the image of  $p$ in $I^{j} M/I^{j+1} M$, and let $0^*=0$. Set $\mathcal{R}(I) := \bigoplus_{n\geqslant0} I^nt^n$, the Rees ring, and $\hat{\mathcal{R}}(I) := \bigoplus_{n\in \mathbb{Z}} I^nt^n$ is the extended Rees ring  of $A$ with respect to $I$, where $ I^n = A $ for every $ n \leqslant 0 $. Set $\mathcal{R}(I,M) := \bigoplus_{n\geqslant 0} I^n M t^n$, the Rees module, and $\hat{\mathcal{R}}(I,M) := \bigoplus_{n\in \mathbb{Z}} I^nMt^n$ is the extended Rees Module of $M$ with respect to $I$. Let $G_{I}(A) := \bigoplus_{n\geqslant 0} I^{n} /I^{n+1} $ be the associated graded ring of $A$ with respect to $I$, and $G_{I}(M) := \bigoplus_{n\geqslant 0} I^{n} M/I^{n+1}M $ be the associated graded module of $M$ with respect to $I$. Throughout this article, we denote  the ideal $\bigoplus_{n\geqslant1} I^{n}t^n$ of $\mathcal{R}(I)$ by $R_+$, and the ideal $\bigoplus_{n\geqslant 1} I^{n} /I^{n+1}$ of $G_I(A)$ by $G_+$.
\s
Set  $L^I(M) := \bigoplus_{n\geqslant0}  M/I^{n+1}M$. The $A$-module $L^I(M)$ can be given an $\mathcal{R}(I)$-module structure as follows. The Rees ring $\mathcal{R}(I)$ is a subring of $\hat{\mathcal{R}}(I)$, and $\hat{\mathcal{R}}(I)$ is a subring of $ S := A[t,t^{-1}] $. So $S$ is an $\mathcal{R}(I)$-module. Therefore $M[t,t^{-1}]=\bigoplus_{n\in \mathbb{Z}}Mt^n=M\otimes_A S$ is an $\mathcal{R}(I)$-module. The exact sequence
\begin{equation} \label{1}
0 \longrightarrow \hat{\mathcal{R}}(I,M) \longrightarrow M[t,t^{-1}] \longrightarrow L^I(M)(-1) \longrightarrow 0
\end{equation}
defines an $ \mathcal{R}(I )$-module structure on $L^I(M)(-1)$, and hence on $L^I(M)$.

The following result is well-known and easy to prove.

\begin{lemma}\label{op}
	Let $ R=\bigoplus_{i\geqslant0}R_{i}$ be a graded ring. Let $E$ be a graded $R$-module  {\rm (}need not be finitely generated{\rm )}. Then the following statements hold true:
	\begin{enumerate}[{\rm (i)}]
		\item If $E_{n} = 0 $ for all $n \gg 0$, and there is an injective graded homomorphism $E(-1) \hookrightarrow E $, then $E = 0$.
		\item If $E_{n} = 0 $ for all $n \ll 0$, and there is an injective graded homomorphism $E \hookrightarrow E(-1) $, then $E = 0$.
	\end{enumerate}
\end{lemma}
\s
Let $ R = \bigoplus_{i \geqslant 0}R_{i} $ be a graded ring. Let $E$ be a graded $R$-module. Let $l$ be a positive integer. The $l$th Veronese subring of $R$ is defined by $R^{<l>} := \bigoplus_{n\geqslant0}R_{nl}$, and the $l$th Veronese submodule of $E$ is defined to be $E^{<l>} := \bigoplus_{n\in \mathbb{Z}}E_{nl}$. It can be observed that $E^{<l>}$ is a graded $R^{<l>}$-module.

\begin{remark}\label{rmkk}{~}
	\begin{enumerate}[{\rm (i)}]
		\item $L^I(M)(-1)$ behaves well with respect to the Veronese functor. It can be easily checked that
		$$L^I(M)(-1)^{<l>} = L^{I^l}(M)(-1).$$
		\item \cite[Proposition~2.5]{Sh99} Veronese functor commutes with local cohomology: Let $ J $ be a homogeneous ideal of $ R $. Then, for every $ i \geqslant 0 $, we have
		\[
		\left( H_J^i(E) \right)^{<l>} \cong H_{J^{<l>}}^i(E^{<l>}) \mbox{ as graded $ R^{<l>} $-modules}.
		\]
	\end{enumerate}
\end{remark}

Although, in general, $ L^I(M) $ is not finitely generated as an $ \mathcal{R}(I) $-module, but it has the following vanishing property.

\begin{lemma}\label{oo}
	Suppose $ \operatorname{grade}(I,M) = g > 0 $.  Then, for every $0 \leqslant i \leqslant g-1$, $ H^i_{R+}(L^I(M))_n = 0 $ for all $n \gg 0$.
\end{lemma}

\begin{proof}
	Since $\operatorname{grade}(I,M)=g > 0$, there exists an $M$-regular sequence $x_1,\ldots,x_g$ in $I$. It can be observed that $x_1t,\ldots,x_gt$ $\in \mathcal{R}(I)_1$ becomes an $M[t,t^{-1}]$-regular sequence. So $H^i_{R+}(M[t,t^{-1}])=0$ for $0\leqslant i \leqslant g - 1$. Therefore, in view of the short exact sequence \eqref{1}  and using the corresponding long exact sequence in local cohomology, we get that
	\begin{align}
	H^i_{R+}(L^I(M)(-1)) & \cong H^{i+1}_{R+}(\hat{\mathcal{R}}(I,M)) ~ \mbox{ for $0\leqslant i\leqslant g-2$, and}\label{ik1}\\
	H^{g-1}_{R+}(L^I(M)(-1)) & \subseteq H^g_{R+}(\hat{\mathcal{R}}(I,M)).\label{ik2}
	\end{align}
	Set $U := \bigoplus_{n<0} Mt^n$. Since $U$ is $R_+$ torsion, we have $H^0_{R+}(U)=U$ and $H^i_{R_+}(U)=0$ for all $i \geqslant 1$. Considering the short exact sequence of $\mathcal{R}(I)$-modules
	\begin{equation*}
	0 \longrightarrow \mathcal{R}(I,M)\longrightarrow \hat{\mathcal{R}}(I,M) \longrightarrow \bigoplus_{n<0} Mt^n \longrightarrow 0,
	\end{equation*}
	the corresponding long exact sequence in local cohomology yields the exact sequence
	\begin{align}
	& 0\longrightarrow U\longrightarrow H^1_{R_+}(\mathcal{R}(I,M)) \longrightarrow H^1_{R+}(\hat{\mathcal{R}}(I,M)) \longrightarrow 0, \label{dg1}\\
	& \mbox{and } H^i_{R_+}(\mathcal{R}(I,M)) \cong H^i_{R_+}(\hat{\mathcal{R}}(I,M)) \mbox{ for } i \geqslant 2. \label{dg2}
	\end{align}
	It is well-known that for each $i \geqslant 0$, $H^i_{R_+}(\mathcal{R}(I,M))_n=0$ for all $n\gg0$. Therefore, in view of \eqref{dg1} and \eqref{dg2}, for each $i \geqslant 0$, $H^i_{R_+}(\hat{\mathcal{R}}(I,M))_n = 0$ for all $n \gg 0$. Hence the lemma follows from \eqref{ik1} and \eqref{ik2}.
\end{proof}
\s
The {\it Ratliff-Rush closure} of $ M $ with respect to $ I $ is defined to be
\[
\widetilde{I M} := \bigcup_{m \geqslant 1} (I^{m+1}M :_M I^m).
\]
It is shown in \cite[Proposition~2.2.(iv)]{Rez3} that if $ \operatorname{grade}(I,M) > 0 $, then $\widetilde{I^{n}M}=I^nM$ for all $ n \gg 0 $. This motivates the following definition:
\[
\rho^I(M) := \min\{n : \widetilde{I^{i}M}=I^iM ~\mbox{for all} ~i\geqslant n \}.
\]
We call $\rho^I(M)$ the {\it Ratliff-Rush number} of $M$ with respect to $I$.

\s\label{mi}
Let $I = (x_1,\ldots, x_m)$. Set $S := A[X_1,\ldots, X_m]$ with deg$A=0$ and deg$X_i=1$ for $i=1,\ldots ,m$. Then $S = \bigoplus_{n\geqslant 0}S_n$, where $S_n$ is the collection of all homogeneous polynomials of degree $n$. So $A = S_0$. We denote the ideal $\bigoplus_{n\geqslant 1}S_n$ of $S$ by $S_+$. We have a surjective homogeneous homomorphism of $A$-algebras, namely $\varphi: S\rightarrow \mathcal{R}(I)$, where $ \varphi(X_i) = x_i t $. We also have the natural map  $\psi : \mathcal{R}(I ) \to G_I(A)$. Note that
\[
\varphi(S_+)=R_+, \quad \psi(R_+) = G_+ \quad \mbox{and} \quad \psi \circ \varphi(S_+)=G_+.
\]
By graded independence theorem (\cite[13.1.6]{BS60}), it does not matter which ring we use to compute local cohomology. So now onwards, we simply use $H^i(-)$ instead of $H^i_{R_+}(-)$ or $H^i_{G_+}(-)$.
\s
The natural map $ 0 \rightarrow {I^nM}/{I^{n+1}M} \rightarrow M/{I^{n+1}M} \rightarrow M/{I^{n}M} \rightarrow 0 $ induces the first fundamental exact sequence (as in \cite[(5)]{tp07}) of $\mathcal{R}(I)$-modules:
\begin{equation}\label{1st}
0 \longrightarrow G_I(M) \longrightarrow L^I(M) \longrightarrow L^I(M)(-1) \longrightarrow 0.
\end{equation}

\s
Let $x$ be an $M$-superficial element with respect to $I$. Set $N=M/{xM}$. For every $ n \geqslant 1 $, we have an exact sequence of $A$-modules:
\[
0 \longrightarrow \dfrac{(I^{n+1}M :_M x)}{I^nM} \longrightarrow \dfrac{M}{I^nM} \stackrel{\psi_n}{\longrightarrow} \dfrac{M}{I^{n+1}M} \longrightarrow \dfrac{N}{I^{n+1}N} \longrightarrow 0,
\]
where $\psi_n(m+I^nM)=xm+I^{n+1}M$ for $ m \in M $. These sequences induce the second fundamental exact sequence (as in \cite[6.2]{tp07})  of $ \mathcal{R}(I) $-modules:
\begin{equation}\label{2nd}
0\longrightarrow B^{I}(x,M) \longrightarrow L^I(M)(-1) \stackrel{\Psi_{xt}}{\longrightarrow} L^I(M) \stackrel{\rho}{\longrightarrow} L^{I}(N) \longrightarrow 0,
\end{equation}
where $\Psi_{xt}$ is multiplication by $ xt \in \mathcal{R}(I)_1 $, and
\[ B^{I}(x,M) := \bigoplus_{n \geqslant 0}(I^{n+1}M:_{M}x)/{I^{n}M}. \]

\s\label{pri}
It is shown in \cite[Proposition~4.7]{tp07} that if $\operatorname{grade}(I, M) > 0$, then
\[
H^0_{R+}(L^I(M)) \cong \bigoplus^{\rho^{I}(M)-1}_{i=0}~ \dfrac{\widetilde{I^{i+1}M}}{I^{i+1}M}.
\]
\s\label{mod-reg-tony}
Let $ x \in I \smallsetminus I^2 $. If $ x^* $ is $G_I(M)$-regular, then $G_I(M)/x^* G_I(M) \cong G_I(M/xM)$ (the proof in \cite[Theorem 7]{hilbert} generalizes in this context).

We now show that $ \operatorname{grade}(G_+, G_I(M)) $ is always bounded by $ \operatorname{grade}(I,M) $.

\begin{lemma}\label{hilbSyz}
	We have that $ \operatorname{grade}(G_+, G_I(M)) \leqslant \operatorname{grade}(I,M) $.
\end{lemma}

\begin{proof}
	We prove the result by induction on $ g := \operatorname{grade}(I,M) $. Let us first consider the case $ g = 0 $. If possible, suppose $ \operatorname{grade}(G_+, G_I(M)) \geqslant 1 $. Then there is a $ G_I(M) $-regular element $ u = x+I^2 \in G_1 $ for some $ x \in I $. Since $ \operatorname{grade}(I,M) = 0 $, $x$ cannot be $M$-regular, i.e., there exists $ a \neq 0 $ in $ M $ such that  $ xa = 0 $. By Krull's Intersection Theorem, there exists $ c \geqslant 0 $ such that $ a \in  I^{c} M  \smallsetminus I^{c+1}M$. Then $ a^* \neq 0 $ in $ I^{c} M / I^{c+1}M $, but $ u a^* = xa + I^{c+2}M = 0 $ yields that $ a^* = 0 $, which is a contradiction. Therefore $ \operatorname{grade}(G_+, G_I(M)) = 0 $.
	
	We assume the result for $ g = l - 1 $, and prove it for $ g = l $ $ (\geqslant 1) $. If possible, suppose that the result is not true for $ g = l $, i.e,  $ \operatorname{grade}(I,M) = l $ and $ \operatorname{grade}(G_+, G_I(M)) \geqslant l+1 $. Then there exists a $G_I(M)$-regular sequence $ u_1, \ldots, u_{l+1} \in G_1 $, where $ u_i = x_i + I^2 $ for some $ x_i \in I $, $ 1 \leqslant i \leqslant l+1 $. By applying a similar procedure as above, one obtains that $ x_1 $ is $ M $-regular. We note that $\operatorname{grade}(I,M/{x_1M})=l-1$, but $u_2,\ldots,u_{l+1}$ is regular on $ {G_I(M)}/{x_1^* G_I(M)} \cong G_I(M/{x_1M})$; see \ref{mod-reg-tony}. This contradicts our induction hypothesis.
\end{proof}

The result below gives a relationship between the first few local cohomologies of $ L^I(M) $ and that of $ G_I(M) $.

\begin{theorem}\label{jc}
	Suppose $\operatorname{grade}(I,M) = g > 0$. Then, for $s \leqslant g-1$, we have $H^{i}(L^I(M)) =0$ for all $0 \leqslant i \leqslant s$ if and only if $H^{i}(G_I(M)) =0$  for all $0 \leqslant i \leqslant s$.
\end{theorem}

\begin{proof}
	In view of the short exact sequence \eqref{1st} and the corresponding long exact sequence in local cohomology, it follows that if $H^i(L^I(M)) =0$  for $i=0,\ldots,s$, then $H^{i}(G_I(M)) =0$ for $i=0,\ldots,s$.	We now prove the converse part by using induction on $s$.	For $ s = 0 $, let us assume that $H^{0}(G_I(M)) =0$. Then \eqref{1st} yields an injective graded homomorphism $ H^0(L^I(M)) \hookrightarrow H^0(L^I(M))(-1) $. Hence, in view of Lemma~\ref{op}.(ii), we obtain that $ H^0(L^I(M)) = 0 $.
	
	We now assume the result for $ s = l - 1 $, and prove it for $ s = l $, where $ l \geqslant 1 $. Let $H^{i}(G_I(M))=0$ for $0 \leqslant i \leqslant l$. So $ \operatorname{grade}(G_+,G_I(M)) \geqslant l + 1 $. Then there is $ x \in I \smallsetminus I^2 $ such that $x^*$ is $G_I(M)$-regular. Hence it can be easily shown that $ (I^{n+1}M :_M x) = I^n M $ for all $ n \geqslant 0 $. In particular, we have $ B^I(x,M) = 0 $ and $ x $ is $M$-superficial. Set $ N := M/{xM} $. Note that ${G_I(M)}/{x^*G_I(M)} \cong G_I(N)$ (see \ref{mod-reg-tony}). So $ \operatorname{grade}(G_+,G_I(N)) \geqslant l $, and hence $H^{i}(G_I(N))=0$ for $0\leqslant i\leqslant l-1$. Therefore, by induction hypothesis, we have $H^{j}(L^I(N))=0$ for $0\leqslant j\leqslant l-1$. Since $B^I(x,M)=0$, the short exact sequence \eqref{2nd} and the corresponding long exact sequence in local cohomology provide us the exact sequences:
	\begin{equation}\label{uvw}
		0 \longrightarrow H^{i}(L^I(M))(-1) \longrightarrow H^{i}(L^I(M)) \quad \mbox{for } 0 \leqslant i \leqslant l.
	\end{equation}
	In view of Lemma~\ref{hilbSyz}, $ \operatorname{grade}(I,M) \geqslant \operatorname{grade}(G_+, G_I(M)) \geqslant l + 1 $. Hence, by Lemma~\ref{oo}, for every $ 0 \leqslant i \leqslant l $, $ H^i(L^I(M))_n = 0 $ for all $ n \gg 0 $. Therefore it follows from \eqref{uvw} and Lemma~\ref{op}(i) that $H^{i}(L^I(M))=0$ for all $0 \leqslant i \leqslant l$.
\end{proof}

As a consequence of Theorem~\ref{jc}, we obtain the following characterization of $ \operatorname{grade}(G_+, G_I(M)) $ in terms of local cohomology of $ L^I(M) $.

\begin{corollary}\label{count-sup}
	Suppose $ \operatorname{grade}(I,M) = g > 0 $. Then
	\[
	\operatorname{grade}(G_+, G_I(M)) = \min\{ i : H^i(L^I(M)) \neq 0, \mbox{ where }0 \leqslant i \leqslant g \}.
	\]
\end{corollary}

\begin{proof}
	It is well-known that
	\begin{equation*}
	\operatorname{grade}(G_+, G_I(M)) = \min\{ i : H_{G_+}^{i}(G_I(M)) \neq 0 \}.
	\end{equation*}
	By Lemma~\ref{hilbSyz}, we have $ \operatorname{grade}(G_+, G_I(M)) \leqslant g $. Set
	\begin{equation*}
	\alpha := \min\{ i : H^i(L^I(M)) \neq 0, \mbox{ where }0 \leqslant i \leqslant g \}.
	\end{equation*}
	By considering \eqref{1st}, it can be easily observed that $ H^i(L^I(M)) \neq 0 $ for some $ i $ with $ 0 \leqslant i \leqslant \operatorname{grade}(G_+, G_I(M)) $ $ (\leqslant g) $. So $ \alpha \leqslant \operatorname{grade}(G_+, G_I(M)) $. Hence, by virtue of Theorem~\ref{jc}, it follows that $ \alpha = \operatorname{grade}(G_+, G_I(M)) $.
\end{proof}

\section{Asymptotic grade for associated graded modules}\label{sec3}

In the present section, we explore the asymptotic behaviour of the associated graded modules for powers of an ideal. We particularly study its grade with respect to the irrelevant ideals of associated graded rings.

Throughout this section, we work with the following hypothesis, but we do not need Cohen-Macaulay assumption everywhere.

\begin{hypothesis}\label{hyp-sec-3}
	Let $(A,\mathfrak{m})$ be a Cohen-Macaulay local ring with infinite residue field, and $ M $ be a Cohen-Macaulay $ A $-module. Let $ I $ be an ideal of $ A $ such that $ \operatorname{grade}(I,M) = g  > 0 $.
\end{hypothesis}

\s[\bf A few invariants]\label{invariants}
In our study, we use the following invariants.
\begin{enumerate}[{\rm (i)}]
	\item
	$ \xi_I(M) := \min \{ g, i : H^i(L^I(M))_{-1} \neq 0, \mbox{ or } H^i(L^I(M))_j \neq 0 \mbox{ for infinitely}\\ \mbox{many } j < 0, \mbox{ where $ i $ varies in } 0 \leqslant i \leqslant g-1 \}$.
	Note that $ 1 \leqslant \xi_I(M) \leqslant g $.
	\item
	The {\it amplitude} of $ M $ with respect to $ I $ is defined to be
	\[
	\operatorname{amp}_{I}(M) := \max \{ |n| : H^i(L^I(M))_{n-1} \neq 0 \mbox{ for some } 0 \leqslant i \leqslant \xi_I(M)-1 \}.
	\]
	It follows from (i) and Lemma~\ref{oo} that $ \operatorname{amp}_{I}(M) < \infty $.
	\item 
	Let $N$ be a graded module {\rm (}not necessarily finitely generated{\rm )}. Define
	\[
	\operatorname{end}(N) := \sup\{ n \in \mathbb{Z} : N_n \neq 0 \}.
	\]
	\item
	By Lemma~\ref{oo}, for every $ 0 \leqslant i \leqslant g - 1 $, $ H^i_{R+}(L^I(M))_n = 0 $ for all $ n \gg 0 $. So we set
	\[
	b^I_i(M) := \operatorname{end}\left(H^i_{R+}(L^I(M))\right) \mbox{ for every } 0 \leqslant i \leqslant g - 1.
	\]
\end{enumerate}

We start by showing a special property of the first local cohomology of $ L^I(M) $.

\begin{lemma}\label{crucial}
	For a fixed integer $ c < 0 $, the following conditions are equivalent:
	\begin{enumerate}[\rm (i)]
		\item
		$ H^1(L^I(M))_c = 0 $.
		\item
		$ H^1(L^I(M))_j = 0 $ for all $ j \leqslant c $.		
	\end{enumerate}
\end{lemma}

\begin{proof}
	We only need to prove (i) $ \Rightarrow $ (ii).
	Suppose $ H^1(L^I(M))_c = 0 $. Let $ x $ be an $ M $-superficial element with respect to $ I $. Then $ (I^{n+1}M :_M x) = I^nM $ for every $ n \gg 0 $, i.e, $ B^I(x,M) $ is $ G_+ $ torsion. Therefore $ H^0(B^I(x,M)) = B^I(x,M) $, and $ H^i(B^I(x,M)) = 0 $ for all $ i \geqslant 1 $. Hence, by splitting \eqref{2nd} into two short exact sequences, and considering the corresponding long exact sequences, one obtains the following exact sequence:
	\begin{align}
	0 \rightarrow B^I(x,M) \longrightarrow & H^0(L^I(M))(-1) \longrightarrow H^0(L^I(M)) \longrightarrow H^0(L^I(N)) \label{les}\\
	\longrightarrow & H^1(L^I(M))(-1) \longrightarrow H^1(L^I(M)) \longrightarrow H^1(L^I(N)),\nonumber
	\end{align}
	where $ N = M/xM $.	Therefore, for every $ n < 0 $, since $ H^0(L^I(N))_n = 0 $, we have the following exact sequence:
	\[
	0 \longrightarrow H^1(L^I(M))_{n-1} \longrightarrow H^1(L^I(M))_n.
	\] 
	Hence, since $ H^1(L^I(M))_c = 0 $, it follows that $ H^1(L^I(M))_j = 0 $ for all $ j \leqslant c $.
\end{proof}

In \cite[Theorem~3.4]{Sh99}, Huckaba and Marley proved that if $ A $ is Cohen-Macaulay with $ \dim(A) \geqslant 2 $, and $I$ is a normal ideal with $ \operatorname{grade}(I) \geqslant 1 $, then $ \operatorname{depth}(G_{I^n}(A)) \geqslant 2 $ for all $n\gg0$. A similar result for $ \operatorname{grade}(G_{I^n}(A)_+, G_{I^n}(A)) $ is shown here.

\begin{theorem}\label{RR}
	Let $ I $ be a normal ideal of $A$ with $ \operatorname{grade}(I) \geqslant 2 $. Also assume that $A$ is excellent.	
	Then
	\[
	\operatorname{grade}(G_{I^n}(A)_+, G_{I^n}(A)) \geqslant 2 \quad \mbox{for all } \; n \gg 0.
	\]
\end{theorem}

\begin{proof}
	Set $ u := \max\{ b^I_0(A), b^I_1(A) \} + 2 $. Let $ l \geqslant u $. We write $ H_{R_+}^i(L^I(A)) = \bigoplus_{n \in \mathbb{Z}} V^i_n $ as it is a graded $\mathcal{R}(I)$-module. It can be observed that $ V_{nl-1}^i = 0 $ for all $ n \geqslant 1 $ and $ i = 0, 1 $. We note that
	\begin{align}
	H_{\mathcal{R}(I^l)_+}^i\left( L^{I^l}(A) \right) (-1) & \cong H_{\mathcal{R}(I^l)_+}^i\left( L^{I^l}(A)(-1) \right) \nonumber\\
	& \cong H_{(R_+)^{<l>}}^i \left( \left( L^I(A)(-1) \right)^{<l>} \right) \mbox{ [by Remark~\ref{rmkk}.(i)]} \label{rmk2.4}\\
	& \cong \left( H_{R_+}^i \left( L^I(A)(-1) \right) \right)^{<l>} \mbox{ [by Remark~\ref{rmkk}.(ii)]} \nonumber\\
	& \cong \bigoplus_{n \in \mathbb{Z}} V^i_{nl - 1}.\nonumber
	\end{align}
	Therefore, for every $ i \in \{ 0, 1 \} $, since $ V_{nl-1}^i = 0 $ for all $ n \geqslant 1 $, we have
	\begin{equation}\label{h0}
	H_{\mathcal{R}(I^l)_+}^i ( L^{I^l}(A) )_n = 0, \mbox{ i.e., }H_{\mathcal{R}(K)_+}^i \big( L^{K}(A) \big)_n = 0 \mbox{ for all } n \geqslant 0,
	\end{equation}
	where $ K := I^l $.	In particular, it follows that $ H_{\mathcal{R}(K)_+}^0(L^K(A)) = 0 $. We now show that $ H_{\mathcal{R}(K)_+}^1(L^K(A)) = 0 $. Note that $K$ is integrally closed. Therefore, by virtue of \cite[Theorem~2.1]{HU14}, after a flat extension, there exists  a superficial element $ x \in K $ such that the ideal $ J := K/(x) $ is integrally closed in $ B := A/(x) $. In view of a sequence like \eqref{les}, by applying \eqref{h0}, we obtain that $ H_{\mathcal{R}(K)_+}^0(L^K(B))_n = 0 $ for all $ n \geqslant 1 $. Hence, by \ref{pri}, we have $ H_{\mathcal{R}(K)_+}^0(L^K(B)) \cong \widetilde{J}/J = 0 $ as $ J $ is integrally closed; see \cite[2.3.3]{RR78}. Therefore, for every $ n $, a sequence like \eqref{les} yields the following exact sequence:
	\begin{equation}\label{0hh}
	0 \longrightarrow H_{\mathcal{R}(K)_+}^1(L^K(A))_{n-1} \longrightarrow H_{\mathcal{R}(K)_+}^1(L^K(A))_{n}.
	\end{equation}
	Since $ H_{\mathcal{R}(K)_+}^1(L^K(A))_n = 0 $ for all $ n \geqslant 0 $, it can be proved by repeatedly applying \eqref{0hh} that $ H_{\mathcal{R}(K)_+}^1(L^K(A))_n = 0 $ for all $ n $, and hence $ H_{\mathcal{R}(K)_+}^1(L^K(A)) = 0 $.	Thus, by virtue of Corollary~\ref{count-sup}, we have that $ \operatorname{grade}(G_{I^l}(A)_+,G_{I^l}(A)) \geqslant 2 $, and this holds true for every $ l \geqslant u $, which completes the proof of the theorem.
\end{proof}

\begin{remark}
	We have used \cite[Theorem~2.1]{HU14} crucially in the proof above. This is the only place where we need that the ring is excellent.
\end{remark}

Since $ \operatorname{grade}(I^n) = \operatorname{grade}(I) $ for every $ n \geqslant 1 $, as an immediate consequence of Theorem~\ref{RR} and Lemma~\ref{hilbSyz}, one obtains the following result.

\begin{corollary}\label{cor-grade-2}
	Let $ I $ be a normal ideal of $A$ with $ \operatorname{grade}(I) = 2 $. Also assume that $A$ is excellent. Then
	\[
	\operatorname{grade}(G_{I^n}(A)_+, G_{I^n}(A)) = 2 \mbox{ for all } n \gg 0.
	\]
\end{corollary}

The following theorem gives an asymptotic lower bound of grade of associated graded modules for powers of an ideal.

\begin{theorem}\label{1p}
	For each $ l > \operatorname{amp}_{I}(M) $, $ \operatorname{grade}(G_{I^l}(A)_+, G_{I^l}(M)) \geqslant \xi_{I}(M) $.
\end{theorem}

\begin{proof}
	Set $ E^i := H^{i}(L^{I}(M)(-1)) $, and $ u := \xi_{I}(M) $. Fix an arbitrary $ l > \operatorname{amp}_{I}(M) $. Also fix $ i $ with $ 0 \leqslant i \leqslant u - 1 $. Then, for $ n \neq 0 $, we have $ E^{i}_{nl} = H^i(L^I(M))_{nl-1} = 0 $ as $ |n| l \geqslant l > \operatorname{amp}_{I}(M) $. Also $ E^{i}_{0} = H^i(L^I(M))_{-1} = 0 $ since $ 0 \leqslant i \leqslant \xi_{I}(M) - 1 $. Hence $(E^{i})^{<l>} = \bigoplus_{n\in \mathbb{Z}} E^{i}_{nl}=0$. So, by Remark~\ref{rmkk} (as in \eqref{rmk2.4}), it follows that
	\begin{equation}\label{vero-loc}
	H_{\mathcal{R}(I^l)_+}^i \left( L^{I^l}(M)(-1) \right) = \left( H_{R_+}^{i} \left(L^{I}(M)(-1) \right) \right)^{<l>} = 0.
	\end{equation}
	Therefore $ H_{R(I^l)_+}^i \left( L^{I^l}(M) \right) = 0 $ for all $ 0 \leqslant i \leqslant u - 1 $. Hence, by virtue of Corollary~\ref{count-sup}, $\operatorname{grade}(G_{I^l}(A)_+, G_{I^l}(M)) \geqslant u = \xi_{I}(M)$  for all $ l > \operatorname{amp}_{I}(M) $.
\end{proof}

The following corollary shows that how the vanishing of a single component of certain local cohomology plays a crucial role in the study of grade of asymptotic associated graded modules.

\begin{corollary}\label{id}
	The following conditions are equivalent:
	\begin{enumerate}[{\rm (i)}]
		\item $ H^1(L^I(M))_{-1} = 0 $.
		\item $\operatorname{grade}(G_{I^l}(A)_+, G_{I^l}(M)) \geqslant 2 $ for all $ l > \operatorname{amp}_{I}(M) $.
		\item $\operatorname{grade}(G_{I^l}(A)_+, G_{I^l}(M)) \geqslant 2 $ for some $ l \geqslant 1 $.
	\end{enumerate}
\end{corollary}

\begin{proof}
	(i) $ \Rightarrow $ (ii): Let $ H^1(L^I(M))_{-1} = 0 $. So, by Lemma~\ref{crucial}, $ H^1(L^I(M))_j = 0 $ for all $j\leqslant-1$. Therefore, since $ \xi_{I}(M) \geqslant 1$ (always), it follows that $ \xi_{I}(M) \geqslant 2$. Hence, in view of Theorem~\ref{1p}, $\operatorname{grade}(G_{I^l}(A)_+, G_{I^l}(M))\geqslant2$ for all $ l > \operatorname{amp}_{I}(M) $.
	
	(ii) $ \Rightarrow $ (iii): It holds trivially.
	
	(iii) $ \Rightarrow $ (i): Suppose $ \operatorname{grade}(G_{I^l}(A)_+, G_{I^l}(M)) \geqslant 2 $ for some $ l \geqslant 1 $. Then it follows from Corollary~\ref{count-sup} that $H_{\mathcal{R}(I^l)_+}^1 \big( L^{I^l}(M) \big) = 0 $. Therefore, as in \eqref{vero-loc}, we obtain that $ H_{R_+}^{i} \left( L^{I}(M)(-1) \right)^{<l>} = 0 $, and hence its $ 0 $th component provides us $ H^1(L^I(M))_{-1} = 0 $.
\end{proof}

As a consequence, we obtain the following asymptotic behaviour of associated graded modules for powers of an ideal.

\begin{corollary}\label{thm-xi-1}
	Exactly one of the following alternatives must hold true:
	\begin{enumerate}[{\rm (i)}]
		\item $ \operatorname{grade}(G_{I^n}(A)_+, G_{I^n}(M)) = 1 $ for all $ n > \operatorname{amp}_{I}(M) $.
		\item $ \operatorname{grade}(G_{I^n}(A)_+, G_{I^n}(M)) \geqslant 2 $ for all $ n > \operatorname{amp}_{I}(M) $.
	\end{enumerate}
\end{corollary}

\begin{proof}
	Since $ \xi_{I}(M) \geqslant 1 $, by virtue of Theorem~\ref{1p}, $ \operatorname{grade}(G_{I^n}(A)_+, G_{I^n}(M)) \geqslant 1$ for all $ n > \operatorname{amp}_{I}(M) $. Hence the result follows from Corollary~\ref{id}.
\end{proof}

Here we prove our main result of this section.

\begin{theorem}\label{thm-xi-2}
	With Hypothesis~{\rm \ref{hyp-sec-3}}, suppose $ \operatorname{height}(I) \geqslant \dim(A) - 2 $. Then
	\[
	\operatorname{grade}(G_{I^l}(A)_+, G_{I^l}(M)) = \xi_{I}(M) \mbox{ for every } l > \operatorname{amp}_{I}(M).
	\]
\end{theorem}

\begin{proof}
	Set $ u := \xi_{I}(M) $. By virtue of Theorem~\ref{1p}, $ \operatorname{grade}(G_{I^l}(A)_+, G_{I^l}(M)) \geqslant u $ for every $ l > \operatorname{amp}_{I}(M) $. If possible, suppose that $ \operatorname{grade}(G_{I^l}(A)_+, G_{I^l}(M)) > u $  for some $ l > \operatorname{amp}_{I}(M) $. Then, in view of Corollary~\ref{count-sup}, $ H_{\mathcal{R}(I^l)_+}^u \big( L^{I^l}(M) \big) = 0 $. Thus, as in \eqref{vero-loc}, we obtain that $ H_{R_+}^u \left( L^{I}(M)(-1) \right)^{<l>} = 0 $, and hence
	\begin{equation}\label{hu0}
	H^{u}(L^{I}(M))_{nl-1} = 0 \mbox{ for all } n \in \mathbb{Z}.
	\end{equation}
	We note that $ u < \operatorname{grade}(G_{I^l}(A)_+, G_{I^l}(M)) \leqslant g $; see Lemma~\ref{hilbSyz}.
	
	The long exact sequence corresponding to \eqref{1st} provides an exact sequence:
	\begin{equation}\label{uuu}
	H^{u-1}\left( L^I(M) \right)(-1) \to H^u(G_I(M)) \to H^u(L^I(M)) \to H^u\left( L^I(M) \right)(-1) .
	\end{equation}
	Since $ u = \xi_{I}(M) $, it follows from the definition of $ \xi_{I}(M) $ that $ H^{u-1}(L^I(M))_n = 0 $ for all $ n \ll 0 $. Therefore \eqref{hu0} and \eqref{uuu} yield that $ H^u(G_I(M))_{nl-1} = 0 $ for all $ n \ll 0 $. Hence, since $ H^u(G_I(M)) $ is tame (due to \cite[Lemma~4.3]{Bb}), there exists some $ c < 0 $ such that $ H^u(G_I(M))_j = 0 $ for all $ j \leqslant c $. (Note that $ \dim(A/I) \leqslant \dim(A) - \operatorname{height}(I) \leqslant 2 $, and $ G_I(M) $ is a finitely generated graded $ G_I(A) $-module).
	
	Since $ H^u(G_I(M))_j = 0 $ for all $ j \leqslant c $, \eqref{uuu} produces an exact sequence
	\begin{equation}
	0 \longrightarrow H^u(L^I(M))_j \longrightarrow H^u(L^I(M))_{j-1} \mbox{ for every } j \leqslant c.
	\end{equation}
	Therefore, if $ m, n \leqslant c $ are integers such that $ m \leqslant n $, then $ H^u(L^I(M))_n $ can be considered as a submodule of $ H^u(L^I(M))_m $. Using this fact and \eqref{hu0}, one can prove that	$ H^u(L^I(M))_j = 0 $ for all $ j \leqslant n'l $, where $ n' $ is a fixed integer such that $ n'l \leqslant c $. Thus we have $ H^u(L^I(M))_j = 0 $ for all $ j \ll 0 $, and $ H^u(L^I(M))_{-1} = 0 $ by \eqref{hu0}. This contradicts that $ u = \xi_{I}(M) < g $. Therefore $ \operatorname{grade}(G_{I^l}(A)_+, G_{I^l}(M)) = \xi_{I}(M)  $ for every $ l > \operatorname{amp}_{I}(M) $.
\end{proof}

\section{On the sets $ \operatorname{Ass}^{\infty}_I(M) $ and $ T^\infty_1(I, M) $}\label{sec4}

Let $ (A,\mathfrak{m}) $ be a local ring. Let $ I $ be an ideal of $ A $, and $ M $ be a finitely generated $ A $-module. By a result of Brodmann \cite{Mb0}, there exists $ n_0 $ such that $ \operatorname{Ass}_A(M/I^nM) = \operatorname{Ass}_A(M/I^{n_0}M)$ for all $ n \geqslant n_0 $. The eventual constant set (i.e., $ \operatorname{Ass}_A(M/I^{n_0}M) $) is denoted by $ \operatorname{Ass}^{\infty}_I(M) $. In \cite[Theorem~1]{MS}, Melkersson and Schenzel generalized Brodmann's result by proving that for every fixed $ i \geqslant 0 $, the set $ \operatorname{Ass}_A \left( \operatorname{Tor}^A_i(M, A/I^n) \right) $ is constant for all $ n \gg 0 $. We denote this stable value by $ T^\infty_i(I, M) $. Note that $ \operatorname{Ass}^{\infty}_I(M) $ is nothing but $ T^\infty_0(I, M) $. In this section, we mainly study the question that when does $ \mathfrak{m} \in \operatorname{Ass}^{\infty}_I(M) $ (resp. $ T^\infty_1(I, M) $)? Our first result in this direction is regarding the set $ \operatorname{Ass}^{\infty}_I(M) $.

\begin{theorem}\label{sir1}
	Let $(A,\mathfrak{m})$ be a Cohen-Macaulay local, non-regular ring, and $L$ be an MCM $A$-module. Suppose $ M = \operatorname{Syz}^A_1(L) ~(\neq 0) $, and $ M_P $ is free for every $ P \in \operatorname{Spec}(A) \smallsetminus \{ \mathfrak{m}  \}$. Let $ I $ be an ideal of $ A $ such that $ \operatorname{grade}(G_{I^n}(A)_+, G_{I^n}(A)) \geqslant 2 $ for all $ n \gg 0 $.
	\[
	\text{If} \ \mathfrak{m} \notin \operatorname{Ass}^{\infty}_I(M), \text{ then } \operatorname{grade}( G_{I^n}(A)_+, G_{I^n}(M)) \geqslant 2  \text{ for every } n > \operatorname{amp}_I(M).
	\]
\end{theorem}

\begin{proof}
	Since $ L $ is an MCM $ A $-module, every $ A $-regular element is $ L $-regular. By virtue of Lemma~\ref{hilbSyz}, from the given hypotheses, it follows that $ \operatorname{grade}(I,A) > 0 $, and hence $ \operatorname{grade}(I,L) > 0 $. So, by Corollary~\ref{thm-xi-1}, $ \operatorname{grade}( G_{I^n}(A)_+, G_{I^n}(L)) \geqslant 1 $ for all $ n \gg 0 $. Therefore, in view of  Corollary~\ref{count-sup}, we obtain that
	\begin{equation}\label{0L0}
	H_{\mathcal{R}(I^n)_+}^0 \left( L^{I^n}(L) \right) = 0 \quad \mbox{for all } n \gg 0.
	\end{equation}
	Note that $ M $ is an MCM $ A $-module. So as above $ \operatorname{grade}(I, M) > 0 $. 
	
	We have a short exact sequence 
	\begin{equation}\label{io}
	0 \longrightarrow M \longrightarrow F \longrightarrow L \longrightarrow 0,
	\end{equation}	
	where $ F $ is a free $A$-module. For every $ n $, by applying $ (A/{I^n}) \otimes_A - $ on \eqref{io}, we obtain an exact sequence:
	\begin{equation}\label{io1}
	0 \longrightarrow \operatorname{Tor}^A_1(A/{I^n}, L) \longrightarrow {M}/{I^nM} \longrightarrow {F}/{I^nF} \longrightarrow {L}/{I^nL} \longrightarrow 0.
	\end{equation}
	For every $ P \in \operatorname{Spec}(A) \smallsetminus \{ \mathfrak{m}  \}$, since $ M_P $ is free, we get that $ L_P $ is free, and hence  $ \operatorname{Tor}^A_1(A/{I^n}, L)_P = 0$. So $ \operatorname{Ass}_A \left( \operatorname{Tor}^A_1(A/{I^n}, L) \right) \subseteq \{ \mathfrak{m} \} $ for every $ n $. Therefore, since $ \mathfrak{m} \notin \operatorname{Ass}^{\infty}_I (M) $, in view of \eqref{io1}, it can be deduced that $ \operatorname{Ass}_A \left( \operatorname{Tor}^A_1(A/{I^n}, L) \right) = \phi $ (empty set) for all $ n \gg 0 $, and hence there is $ c' \geqslant 1 $ such that $ \operatorname{Tor}^A_1(A/{I^n}, L) = 0 $ for every $ n \geqslant c' $. Thus \eqref{io1} yields an exact sequence:
	\begin{equation}\label{io2}
	0 \longrightarrow {M}/{I^nM} \longrightarrow {F}/{I^nF} \longrightarrow {L}/{I^nL} \longrightarrow 0
	\end{equation}
	for every $ n \geqslant c' $. In particular, for every $ n \geqslant c' $, we have short exact sequences:
	\begin{equation*}
	0 \longrightarrow {M}/{I^{nk}M} \longrightarrow {F}/{I^{nk}F} \longrightarrow {L}/{I^{nk}L} \longrightarrow 0
	\end{equation*}
	for all $ k \geqslant 1 $, which induce an exact sequence of $ \mathcal{R}(I) $-modules:
	\begin{equation}\label{io3}
	0\longrightarrow L^{I^n}(M)(-1) \longrightarrow L^{I^n}(F)(-1) \longrightarrow L^{I^n}(L)(-1) \longrightarrow 0.
	\end{equation}
	The corresponding long exact sequence of local cohomology modules yields
	\begin{align}\label{io4}
	0 \longrightarrow & H_{\mathcal{R}(I^n)_+}^0 \left( L^{I^n}(M) \right) \longrightarrow H_{\mathcal{R}(I^n)_+}^0 \left( L^{I^n}(F) \right) \longrightarrow H_{\mathcal{R}(I^n)_+}^0 \left( L^{I^n}(L) \right)\\
	\longrightarrow & H_{\mathcal{R}(I^n)_+}^1 \left( L^{I^n}(M) \right) \longrightarrow H_{\mathcal{R}(I^n)_+}^1 \left( L^{I^n}(F) \right). \nonumber
	\end{align}
	Since $ \operatorname{grade}(G_{I^n}(A)_+, G_{I^n}(A)) \geqslant 2 $ for all $ n \gg 0 $, by virtue of Corollary~\ref{count-sup}, we get that $ H_{\mathcal{R}(I^n)_+}^i \left( L^{I^n}(A) \right) = 0 $ for $ i = 0, 1 $, and for all $ n \gg 0 $. Therefore
	\begin{equation}\label{io5}
	H_{\mathcal{R}(I^n)_+}^0 \left( L^{I^n}(F) \right) = 0 = H_{\mathcal{R}(I^n)_+}^1 \left( L^{I^n}(F) \right) \quad \mbox{for all } n \gg 0.
	\end{equation}
	It follows from \eqref{0L0}, \eqref{io4} and \eqref{io5} that
	\begin{equation*}
	H_{\mathcal{R}(I^n)_+}^0 \left( L^{I^n}(M) \right) = 0 = H_{\mathcal{R}(I^n)_+}^1 \left( L^{I^n}(M) \right) \quad \mbox{for all } n \gg 0.
	\end{equation*}
	Hence, in view of Corollaries~\ref{count-sup} and \ref{id}, $ \operatorname{grade}(G_{I^n}(A)_+, G_{I^n}(M)) \geqslant 2 $ for every $ n > \operatorname{amp}_I(M) $, which completes the proof of the theorem.
\end{proof}

We now give

\begin{proof}[Proof of Theorem~\ref{first}]
	This follows from Theorem~\ref{sir1} and Theorem~\ref{RR}.
\end{proof}

The following result gives a variation of Theorem~\ref{sir1}.

\begin{theorem}\label{var-sir1}
	Let $ (A,\mathfrak{m}) $ be a Cohen-Macaulay local ring of dimension $ d \geqslant 3 $. Set $ M := \operatorname{Syz}^A_1(L)$ for some MCM $A$-module $ L $. Let $ I $ be a locally complete intersection ideal of $A$ with $ \operatorname{height}(I) = d - 1 $, the analytic spread $ l(I) = d $, and $ \operatorname{grade}(G_{I^n}(A)_+, G_{I^n}(A)) \geqslant 2 $ for all $ n \gg 0 $. We have that
	\[
		\text{if} \ \mathfrak{m} \notin \operatorname{Ass}^{\infty}_I(M), \text{ then } \operatorname{grade}( G_{I^n}(A)_+, G_{I^n}(M)) \geqslant 2  \text{ for every } n > \operatorname{amp}_I(M).
	\]
\end{theorem}

\begin{remark}
	See \cite[2.2]{HHF} for cases when the hypotheses on the ideal $ I $  are satisfied.
\end{remark}

\begin{proof}[Proof of Theorem~\ref{var-sir1}]
	We claim that $ \operatorname{Tor}^A_1(A/{I^n}, L) $ has finite length for all $ n \gg 0 $. To show this, consider $ P \in \operatorname{Spec}(A) \smallsetminus \{ \mathfrak{m} \} $. If $ L_P $ is free, then $ \operatorname{Tor}^A_1(A/{I^n}, L)_P = 0 $. So we may assume that $ L_P $ is not free. If $ I \nsubseteq P $, then also $ \operatorname{Tor}^A_1(A/{I^n}, L)_P = 0 $ for every $ n \geqslant 1 $. So we assume that $ I \subseteq P $. Since $ \operatorname{height}(I) = d - 1 $ and $ P \neq \mathfrak{m} $, we have $ \operatorname{height}(P) = d - 1 $, and hence $ P $ is minimal over $ I $. (So there are finitely many such prime ideals). Note that $ I_P $ is a $ P A_P $-primary ideal of $ A_P $. Since $ I $ is locally complete intersection, $ I_P $ is generated by an $ A_P $-regular sequence of length $ d - 1 $. Therefore, in view of \cite[Remark~20]{hilbert}, we obtain that $ \operatorname{Tor}^A_1(A/{I^n}, L)_P = 0 $ for all $ n \gg 0 $. Hence $ \operatorname{Tor}^A_1(A/{I^n}, L) $ has finite length for all $ n \gg 0 $. Now, along with the same arguments as in the proof of Theorem~\ref{sir1}, it follows that $\operatorname{grade}(G_{I^n}(A)_+, G_{I^n}(M))\geqslant2$ for every $ n > \operatorname{amp}_I(M) $.
\end{proof}
\s 
({\it MCM approximations and an invariant of modules}).
Let $(A,\mathfrak{m})$ be a Gorenstein local ring. Consider a finitely generated $ A $-module $ M $. By virtue of \cite[Theorem~A]{AB89}, there is an MCM approximation of $ M $, i.e., a short exact sequence $ s : 0 \rightarrow Y \rightarrow X \rightarrow M \rightarrow 0 $ of $ A $-modules, where $ X $ is MCM and $ Y $ has finite injective dimension (equivalently, $ Y $ has finite projective dimension since $ A $ is Gorenstein). We say that $ X $ is an MCM approximation of $ M $. In view of \cite[Theorem~B]{AB89}, if $ s^\prime \colon 0 \rightarrow Y^\prime \rightarrow X^\prime \rightarrow M \rightarrow 0$ is another MCM approximation of $ M $, then $ X $ and $ X^\prime $ are stably isomorphic, i.e., there exist finitely generated free $ A $-modules $ F $ and $ G $ such that $ X \oplus F \cong X^\prime \oplus G $, and hence $ \operatorname{Syz}^A_1(X) \cong \operatorname{Syz}^A_1(X') $. Thus $ \operatorname{Syz}^A_1(X) $ is an invariant of $ M $. Note that $ \operatorname{Syz}^A_1(X) = 0 $ if and only if $ \operatorname{projdim}_A(M) $ is finite.

We use the following lemma to prove our result on $ T^\infty_1(I, M) $.

\begin{lemma}\label{oi}
	Let $ (A,\mathfrak{m}) $ be a Gorenstein local ring of dimension $ d $. Suppose $ A $ has isolated singularity. Let $ I $ be a normal ideal of $ A $ such that $ l(I) < d $. Let $ M $ be a Cohen-Macaulay $ A $-module of dimension $ d - 1 $, and $ \operatorname{projdim}_A(M) = \infty $. Let $ X_M $ be an MCM approximation of $ M $. Then the following statements are equivalent:
	\begin{enumerate}[{\rm (i)}]
		\item $ \mathfrak{m} \notin T^\infty_1(I, M) $ {\rm(}i.e., $ \mathfrak{m} \notin \operatorname{Ass}_A(\operatorname{Tor}^A_1(M,A/{I^n})) $ for all $ n \gg 0 ${\rm )}.
		\item $ \operatorname{Tor}^A_1(X_M, A/{I^n}) = 0 $ for all $ n \gg 0 $.
	\end{enumerate}
\end{lemma}

\begin{proof}
	Since $ X_M $ is an MCM approximation of $ M $, and $ \operatorname{depth}(M) \geqslant d - 1 $, there is a short exact sequence $ 0 \to F \to X_M \to M \to 0 $, where $ F $ is a free $ A $-module. The corresponding long exact sequences of Tor-modules yield an exact sequence
	\begin{align}\label{3b}
	0 \longrightarrow \operatorname{Tor}^A_1(X_M,A/{I^n}) \longrightarrow \operatorname{Tor}^A_1(M,A/{I^n}) & \longrightarrow \\
	F/{I^nF} \longrightarrow {X_M}/{I^nX_M} \longrightarrow M/{I^nM} & \longrightarrow 0 \quad \mbox{for every } n \geqslant 1.\nonumber
	\end{align}
	
	(i) $ \Rightarrow $ (ii): Since $ A $ has isolated singularity, it follows that $ (X_M)_P $ is free $ A_P $-module for every $ P \in \operatorname{Spec}(A) \smallsetminus \{ \mathfrak{m} \}$. So $ \operatorname{Tor}^A_1(X_M, A/{I^n}) $ has finite length, and hence $ \operatorname{Ass}_A\left( \operatorname{Tor}^A_1(X_M, A/{I^n}) \right) \subseteq \{ \mathfrak{m} \} $ for every $ n \geqslant 1 $. Therefore, since $ \mathfrak{m} \notin T^\infty_1(I, M) $, in view of \eqref{3b}, we obtain that $ \operatorname{Ass}_A\left( \operatorname{Tor}^A_1(X_M, A/{I^n}) \right) = \phi $ for all $ n \gg 0 $, which implies that $ \operatorname{Tor}^A_1(X_M, A/{I^n}) = 0 $ for all $ n \gg 0 $.
	
	(ii) $ \Rightarrow $ (i): Since $ I $ is normal, and $ l(I) < d $, by virtue of \cite[Proposition~4.1]{Mcada06}, we have $ \mathfrak{m} \notin \operatorname{Ass}^\infty_I(A) $. In view of \eqref{3b}, since $ \operatorname{Tor}^A_1(X_M, A/{I^n}) = 0 $ for all $ n \gg 0 $, it follows that $ T^\infty_1(I, M) \subseteq \operatorname{Ass}^\infty_I(A) $, and hence $ \mathfrak{m} \notin T^\infty_1(I, M) $.
\end{proof}

The following theorem provides us a necessary and sufficient condition for `$ \mathfrak{m} \in T^\infty_1(I, M) $' on certain class of ideals and modules over a Gorenstein local ring.

\begin{theorem}\label{sir2}
	Let $ (A,\mathfrak{m}) $ be an excellent Gorenstein local ring of dimension $ d $. Suppose $ A $ has isolated singularity. Let $ I $ be a normal ideal of $A$ with $ \operatorname{height}(I) \geqslant 2 $ and $ l(I) < d $. Let $ M $ be a Cohen-Macaulay $ A $-module of dimension $ d - 1 $ and $ \operatorname{projdim}_A(M) = \infty $. Let $ X_M $ be an MCM approximation of $ M $. Set $ N := \operatorname{Syz}^A_1(X_M) $. Then the following statements are equivalent:
	\begin{enumerate}[{\rm (i)}]
		\item $ \mathfrak{m} \notin T^\infty_1(I, M) $ {\rm(}i.e., $ \mathfrak{m} \notin \operatorname{Ass}_A(\operatorname{Tor}^A_1(M,A/{I^n})) $ for all $ n \gg 0 ${\rm )}.
		\item $ \mathfrak{m} \notin \operatorname{Ass}^\infty_I(N) $ {\rm(}equivalently, $ \operatorname{depth}(N/{I^nN}) \geqslant 1 $ for all $ n \gg 0 ${\rm )}.
	\end{enumerate}
	Furthermore, if this holds true, then $ \operatorname{grade}(G_{I^n}(A)_+, G_{I^n}(N)) \geqslant 2 $ for every $ n > \operatorname{amp}_I(N) $.
\end{theorem}

\begin{proof}
	Note that $ N = \operatorname{Syz}^A_1(X_M) $ is a non-zero module. We have an exact sequence $ 0 \to N \to G \to X_M \to 0 $, where $ G $ is a free $A$-module. The corresponding long exact sequences of Tor-modules yield an exact sequence (for every $ n \geqslant 1 $):
	\begin{align}\label{3bb}
	0 \longrightarrow \operatorname{Tor}^A_1(X_M,A/{I^n}) \longrightarrow N/{I^nN} \longrightarrow G/{I^nG} \longrightarrow {X_M}/{I^nX_M} \longrightarrow 0.
	\end{align}
	
	(i) $ \Rightarrow $ (ii): Since $ \mathfrak{m} \notin T^\infty_1(I, M) $, by virtue of Lemma~\ref{oi}, $ \operatorname{Tor}^A_1(X_M, A/{I^n}) = 0 $ for all $ n \gg 0 $. Hence \eqref{3bb} yields that $ \operatorname{Ass}^\infty_I(N) \subseteq \operatorname{Ass}^\infty_I(G) = \operatorname{Ass}^\infty_I(A) $. Therefore, since $ \mathfrak{m} \notin \operatorname{Ass}^\infty_I(A) $ (due to \cite[Proposition~4.1]{Mcada06}), we obtain that $ \mathfrak{m} \notin \operatorname{Ass}^\infty_I(N) $.
	
	(ii) $ \Rightarrow $ (i): Since $ A $ has isolated singularity, as in the proof of Lemma~\ref{oi}, it follows that $ \operatorname{Ass}_A\left( \operatorname{Tor}^A_1(X_M, A/{I^n}) \right) \subseteq \{ \mathfrak{m} \} $ for every $ n \geqslant 1 $. Thus, since $ \mathfrak{m} \notin \operatorname{Ass}^\infty_I(N) $, in view of \eqref{3bb}, it can be observed that $ \operatorname{Ass}_A\left( \operatorname{Tor}^A_1(X_M, A/{I^n}) \right) = \phi $ for all $ n \gg 0 $. Equivalently, $ \operatorname{Tor}^A_1(X_M, A/{I^n}) = 0 $ for all $ n \gg 0 $. Hence the implication follows from Lemma~\ref{oi}.
	
	Since $ \operatorname{height}(I) \geqslant 2 $, we have $ \operatorname{grade}(G_{I^n}(A)_+, G_{I^n}(A)) \geqslant 2 $ for all $ n \gg 0 $; see Theorem~\ref{RR}. So the last assertion follows from Theorem~\ref{sir1}.
\end{proof}

\section{Asymptotic prime divisors over complete intersections}\label{sec5}

Let $ A $ be a local complete intersection, and $ M $ be a finitely generated $ A $-module. Suppose either $ I $ is a principal ideal or $ I $ has a principal reduction generated by an $ A $-regular element. In this section, we analyze the asymptotic stability of certain associated prime ideals of Tor-modules $ \operatorname{Tor}_i^A(M, A/I^n) $ if both $ i $ and $ n $ tend to $ \infty $.

\subsection{Module structure on Tor}\label{Module structure on Tor}

We first discuss the graded module structure on direct sum of Tor-modules which we are going to use in order to prove our main results on asymptotic prime divisors of Tor-modules.

Let $Q$ be a ring, and ${\bf f} = f_1, \ldots, f_c$ be a $Q$-regular sequence. Set $A := Q/({\bf f})$. Let $M$ and $N$ be finitely generated $A$-modules.

\s
Let $ \mathbb{F} : \quad \cdots \rightarrow F_n \rightarrow \cdots \rightarrow F_1 \rightarrow F_0 \rightarrow 0 $ be a free resolution of $M$ by finitely generated free $A$-modules. Let
\[
t'_j : \mathbb{F} \longrightarrow \mathbb{F}(-2), \quad 1 \leqslant j \leqslant c
\]
be the {\it Eisenbud operators} defined by ${\bf f} = f_1, \ldots, f_c$; see \cite[Section~1]{Eis80}. In view of \cite[Corollary~1.4]{Eis80}, the chain maps $t'_j$ are determined uniquely up to homotopy. In particular, they induce well-defined maps
\[
t_j : \operatorname{Tor}_i^A(M,N) \longrightarrow \operatorname{Tor}_{i-2}^A(M,N)
\]
(for all $i \in \mathbb{Z}$ and $j = 1,\ldots,c$) on the homology of $ \mathbb{F} \otimes_A N$. In \cite[Corollary~1.5]{Eis80}, it is shown that the chain maps $t'_j$ ($1 \leqslant j \leqslant c$) commute up to homotopy. Thus
\begin{equation*}
\operatorname{Tor}_{\star}^A(M,N) := \bigoplus_{i \in \mathbb{Z}} \operatorname{Tor}_{-i}^A(M,N)
\end{equation*}
turns into a $\mathbb{Z}$-graded $\mathscr{S} := A[t_1,\ldots,t_c]$-module, where $\mathscr{S}$ is the $\mathbb{N}$-graded polynomial ring over $A$ in the operators $t_j$ defined by ${\bf f}$ with $\deg(t_j) = 2$ for all $j = 1,\ldots,c$. Here note that for every $i \in \mathbb{Z}$, the $i$th component of $\operatorname{Tor}_{\star}^A(M,N)$ is $\operatorname{Tor}_{-i}^A(M,N)$. This structure depends only on ${\bf f}$, are natural in both module arguments and commute with the connecting maps induced by short exact sequences.

\subsection{Stability of primes in $ \operatorname{Ass}_A\left( \operatorname{Tor}_i^A(M,N) \right) $}

Here we study the asymptotic stability of certain associated prime ideals of Tor-modules $ \operatorname{Tor}_i^A(M, N) $, $(i \geqslant 0)$, where $ M $ and $ N $ are finitely generated modules over a local complete intersection ring $A$ (see Corollary~\ref{corollary: asymptotic Ass on Tor}).

We denote the collection of all minimal prime ideals in the support of $ M $ by $ \operatorname{Min}_A(M) $ (or simply by $ \operatorname{Min}(M) $). It is well-known that $ \operatorname{Min}(M) \subseteq \operatorname{Ass}_A(M) \subseteq \operatorname{Supp}(M) $. Recall that a local ring $(A,\mathfrak{m})$ is called a {\it complete intersection ring} if its $ \mathfrak{m} $-adic completion $\widehat{A} = Q/({\bf f})$, where $Q$ is a complete regular local ring and ${\bf f} = f_1,\ldots,f_c$ is a $Q$-regular sequence. To prove our results, we may assume that $A$ is complete because of the following well-known fact on associate primes:

\begin{lemma}\label{lemma: Ass: Completion}
	For an $A$-module $M$, we have
	\[
	\operatorname{Ass}_A(M) = \left\{ \mathfrak{q} \cap A : \mathfrak{q} \in \operatorname{Ass}_{\widehat{A}} \left( M \otimes_A \widehat{A} \right) \right\}.
	\]
\end{lemma}
It is now enough to prove our result with the following hypothesis:

\begin{hypothesis}\label{hypothesis 1}
	Let $A = Q/(f_1,\ldots,f_c)$, where $Q$ is a regular local ring, and $ f_1,\ldots,f_c $ is a $Q$-regular sequence. Let $M$ and $N$ be finitely generated $A$-modules.
\end{hypothesis}

We show our result with the following more general hypothesis:

\begin{hypothesis}\label{hypothesis 2}
	Let $A = Q/\mathfrak{a}$, where $Q$ is a regular ring of finite Krull dimension, and $ \mathfrak{a} \subseteq Q $ is an ideal such that $ \mathfrak{a}_{\mathfrak{q}} \subseteq Q_{\mathfrak{q}} $ is generated by a $ Q_{\mathfrak{q}} $-regular sequence for every $ \mathfrak{q} \in \operatorname{Var}(\mathfrak{a}) $. Let $M$ and $N$ be finitely generated $A$-modules.
\end{hypothesis}

It should be noticed that a ring $ A $ satisfies Hypothesis~\ref{hypothesis 1} implies that $ A $ satisfies Hypothesis~\ref{hypothesis 2}. With the Hypothesis~\ref{hypothesis 2}, we have the following well-known bounds for complete intersection dimension and complexity:
\begin{align*}
\operatorname{CI-dim}_A(M) & = \max\{ \operatorname{CI-dim}_{A_{\mathfrak{m}}}(M_{\mathfrak{m}}) : \mathfrak{m} \in \operatorname{Max}(A) \} \quad \mbox{[by definition]} \\
& \leqslant \max\{ \dim(A_{\mathfrak{m}}) : \mathfrak{m} \in \operatorname{Max}(A) \} \quad \mbox{[see, e.g., \cite[4.1.5]{AB00}]} \\
& = \dim(A); \\
\operatorname{cx}_A(M)  & = \max\{ \operatorname{cx}_{A_{\mathfrak{m}}}( M_{\mathfrak{m}} ) :  \mathfrak{m} \in \operatorname{Max}(A) \} \quad \mbox{[by definition]} \\
& \leqslant \max\{ \operatorname{codim}(A_{\mathfrak{m}}) :  \mathfrak{m} \in \operatorname{Max}(A) \} \quad \mbox{[see, e.g., \cite[1.4]{AB00}]} \\
& \leqslant \dim(Q).
\end{align*}
Therefore, by \cite[Theorem~4.9]{AB00}, we have the following result:
\begin{theorem}\label{theorem: vanishing of Tor}
	With the {\rm Hypothesis~\ref{hypothesis 2}}, the following statements are equivalent:
	\begin{enumerate}[{\rm (1)}]
		\item $ \operatorname{Tor}_i^A(M,N) = 0 $ for $ \dim(Q) + 1 $ consecutive values of $ i > \dim(A) $;
		\item $ \operatorname{Tor}_i^A(M,N) = 0 $ for all $ i \gg 0 $;
		\item $ \operatorname{Tor}_i^A(M,N) = 0 $ for all $ i > \dim(A) $.
	\end{enumerate}	
\end{theorem}

Let us recall the following asymptotic behaviour of Tor-modules.

\begin{lemma}\cite[Theorem~3.1]{G} \label{lemma: Dao, *Artinian}
	With the {\rm Hypothesis~\ref{hypothesis 1}}, if $ \lambda_A\left( \operatorname{Tor}_i^A(M,N) \right) $ is finite for all $ i \gg 0 $ {\rm (}where $ \lambda_A(-) $ is the length function{\rm )}, then
	\[
	\bigoplus_{i \ll 0} \operatorname{Tor}_{-i}^A(M,N) \quad\mbox{is a *Artinian graded $A[t_1,\ldots,t_c]$-module},
	\]
	where $ \deg(t_j) = 2 $ for all $  j = 1,\ldots,c $.
\end{lemma}

As a consequence of this lemma, we obtain the following result:

\begin{proposition}\label{proposition: Tor: polynomials}
	Let $ A $ be a local complete intersection ring. Let $ M $ and $ N $ be finitely generated $ A $-modules. If $ \lambda_A\left( \operatorname{Tor}_i^A(M,N) \right) $ is finite for all sufficiently large integer $ i $, then we have that
	\[
	\lambda_A\left( \operatorname{Tor}_{2i}^A(M,N) \right) \;\;\mbox{ and }\;\; \lambda_A\left( \operatorname{Tor}_{2i + 1}^A(M,N) \right)
	\]
	are given by polynomials in $ i $ over $ \mathbb{Q} $ for all sufficiently large integer $ i $.
\end{proposition}

\begin{proof}
	Without loss of generality, we may assume that $A$ is complete. Then the proposition follows from Lemmas~\ref{lemma: Dao, *Artinian} and \ref{lemma: *Artinian, Hilbert function}.
\end{proof}

The following result is a consequence of the graded version of Matlis duality and the Hilbert-Serre Theorem. It might be known for the experts. But we give a proof here for the reader's convenience.

\begin{lemma}\label{lemma: *Artinian, Hilbert function}
	Let $ (A,\mathfrak{m}) $ be a complete local ring. Let $ L = \bigoplus_{i \in \mathbb{Z}} L_i $ be a *Artinian graded $A[t_1,\ldots,t_c]$-module, where $ \deg(t_j) = 2 $ for all $ 1 \leqslant j \leqslant c $, and $ \lambda_A(L_i) $ is finite for all $ i \ll 0 $. Then $ \lambda_A(L_{-2i}) $ and $ \lambda_A(L_{-2i - 1}) $ are given by polynomials in $ i $ over $ \mathbb{Q} $ for all sufficiently large $ i $.
\end{lemma}

\begin{proof}
	We use the graded Matlis duality. Let us recall the following definitions: *complete from \cite[p.\,142]{BH98}; *local from \cite[p.\,139]{BH98}; and $ \mbox{*}\operatorname{Hom}(-,-) $ from \cite[p.\,33]{BH98}. Note that $A[t_1,\ldots,t_c]$ is a Noetherian *complete *local ring. We set $ E := E_A(A/\mathfrak{m}) $, the injective hull of $ A/\mathfrak{m} $. Also set $ L^{\vee} := \mbox{*}\operatorname{Hom}(L,E) $. Notice that $ (L^{\vee})_i = \operatorname{Hom}_A(L_{-i},E) $ for all $ i \in \mathbb{Z} $.
	
	Since $A[t_1,\ldots,t_c]$ is a Noetherian *complete *local ring, by virtue of Matlis duality for graded modules (\cite[3.6.17]{BH98}), we obtain that $ L^{\vee} $ is a finitely generated graded $A[t_1,\ldots,t_c]$-module. Let $ i_0 $ be such that $ \lambda_A(L_{-i}) $ is finite for all $ i \geqslant i_0 $. Hence
	\begin{equation}\label{lemma: *Artinian, Hilbert function: equation 1}
	\lambda_A\left( (L^{\vee})_i \right) = \lambda_A\left( \operatorname{Hom}_A(L_{-i},E) \right) = \lambda_A( L_{-i} )
	\end{equation}
	is finite for all $ i \geqslant i_0 $; see, e.g., \cite[3.2.12]{BH98}. Since $ \bigoplus_{i \geqslant i_0} (L^{\vee})_i $ is a graded $A[t_1,\ldots,t_c]$-submodule of $ L^{\vee} $, we have that $ \bigoplus_{i \geqslant i_0} (L^{\vee})_i $ is a finitely generated graded module over $A[t_1,\ldots,t_c]$. Therefore, by the Hilbert-Serre Theorem, we obtain that $ \lambda_A\left( (L^{\vee})_{2i} \right) $ and $ \lambda_A\left( (L^{\vee})_{2i+1} \right) $ are given by polynomials in $ i $ over $ \mathbb{Q} $ for all $ i \gg 0 $, and hence the lemma follows from \eqref{lemma: *Artinian, Hilbert function: equation 1}.
\end{proof}

We are now in a position to prove our main result of this section.
\begin{theorem}\label{theorem: asymptotic Ass on Tor}
	With the {\rm Hypothesis~\ref{hypothesis 2}}, exactly one of the following alternatives must hold:
	\begin{enumerate}[{\rm (1)}]
		\item $ \operatorname{Tor}_i^A(M, N) = 0 $ for all $ i > \dim(A) $;
		\item There exists a non-empty finite subset $ \mathcal{A} $ of $ \operatorname{Spec}(A) $ such that for every $ \mathfrak{p} \in \mathcal{A} $, at least one of the following statements holds true:
		\begin{enumerate}[{\rm (i)}]
			\item $ \mathfrak{p} \in \operatorname{Min}\left( \operatorname{Tor}_{2i}^A(M, N) \right) $ for all $ i \gg 0 $;
			\item $ \mathfrak{p} \in \operatorname{Min}\left( \operatorname{Tor}_{2i+1}^A(M, N) \right) $ for all $ i \gg 0 $.
		\end{enumerate}
	\end{enumerate}
\end{theorem}
\begin{proof}
	We set
	\[
	\mathcal{B} := \bigcup\left\{ \operatorname{Supp}\left( \operatorname{Tor}_i^A(M,N) \right) : \dim(A) < i \leqslant \dim(A) + \dim(Q) + 1 \right\}.
	\]
	If $ \mathcal{B} = \phi $ (empty set), then $ \operatorname{Tor}_i^A(M,N) = 0 $ for all $ \dim(A) < i \leqslant \dim(A) + \dim(Q) + 1 $, and hence, by virtue of Theorem~\ref{theorem: vanishing of Tor}, we get that $ \operatorname{Tor}_i^A(M,N) = 0 $ for all $ i > \dim(A) $. So we may assume that $ \mathcal{B} \neq \phi $. In this case, we prove that the statement (2) holds true. We denote the collection of minimal primes in $ \mathcal{B} $ by $ \mathcal{A} $, i.e.,
	\begin{equation}\label{theorem: asymptotic Ass on Tor: equation 1}
	\mathcal{A} := \left\{ \mathfrak{p} \in \mathcal{B} : \mathfrak{q} \in \operatorname{Spec}(A) \mbox{ and } \mathfrak{q} \subsetneq \mathfrak{p} \Rightarrow \mathfrak{q} \notin \mathcal{B} \right\}.
	\end{equation}
	Clearly, $ \mathcal{A} $ is a non-empty finite subset of $ \operatorname{Spec}(A) $. We claim that $ \mathcal{A} $ satisfies statement (2) in the theorem. To prove this claim, let us fix an arbitrary $ \mathfrak{p} \in \mathcal{A} $.
	
	If $ \mathfrak{q} \in \operatorname{Spec}(A) $ be such that $ \mathfrak{q} \subsetneq \mathfrak{p} $, then $ \mathfrak{q} \notin \mathcal{B} $, i.e., $ \operatorname{Tor}_i^{A_{\mathfrak{q}}}(M_{\mathfrak{q}}, N_{\mathfrak{q}}) = 0 $ for all $ \dim(A) < i \leqslant \dim(A) + \dim(Q) + 1 $, and hence, in view of Theorem~\ref{theorem: vanishing of Tor}, we obtain that $ \operatorname{Tor}_i^{A_{\mathfrak{q}}}(M_{\mathfrak{q}}, N_{\mathfrak{q}}) = 0 $ for all $ i > \dim(A) $. Therefore
	\begin{equation}\label{theorem: asymptotic Ass on Tor: equation 2}
	\operatorname{Supp}\left( \operatorname{Tor}_i^{A_{\mathfrak{p}}}(M_{\mathfrak{p}}, N_{\mathfrak{p}}) \right) \subseteq \left\{ \mathfrak{p} A_{\mathfrak{p}} \right\} \quad \mbox{for all } i > \dim(A),
	\end{equation}
	which gives
	\begin{equation}\label{theorem: asymptotic Ass on Tor: equation 3}
	\lambda_{A_{\mathfrak{p}}}\left( \operatorname{Tor}_i^{A_{\mathfrak{p}}}(M_{\mathfrak{p}}, N_{\mathfrak{p}}) \right) \quad \mbox{is finite for all } i > \dim(A).
	\end{equation}
	Since $ A_{\mathfrak{p}} $ satisfies Hypothesis~\ref{hypothesis 1}, by Proposition~\ref{proposition: Tor: polynomials}, there are polynomials $ P_1(z) $ and $ P_2(z) $ in $ z $ over $ \mathbb{Q} $ such that
	\begin{align}
	\lambda_{A_{\mathfrak{p}}}\left( \operatorname{Tor}_{2i}^{A_{\mathfrak{p}}}(M_{\mathfrak{p}}, N_{\mathfrak{p}}) \right) & = P_1(i) \quad\mbox{for all } i \gg 0;\label{theorem: asymptotic Ass on Tor: equation 4}\\
	\lambda_{A_{\mathfrak{p}}}\left( \operatorname{Tor}_{2i+1}^{A_{\mathfrak{p}}} (M_{\mathfrak{p}}, N_{\mathfrak{p}}) \right) & = P_2(i) \quad\mbox{for all } i \gg 0.\label{theorem: asymptotic Ass on Tor: equation 5}
	\end{align}
	
	We now show that both $ P_1(z) $ and $ P_2(z) $ cannot be zero polynomials. If this is not the case, then we have
	\[
	\operatorname{Tor}_i^{A_{\mathfrak{p}}}(M_{\mathfrak{p}}, N_{\mathfrak{p}}) = 0 \quad \mbox{for all } i \gg 0,
	\]
	which yields (by Theorem~\ref{theorem: vanishing of Tor}) that
	\[
	\operatorname{Tor}_i^{A_{\mathfrak{p}}}(M_{\mathfrak{p}}, N_{\mathfrak{p}}) = 0 \quad \mbox{for all } i > \dim(A),
	\]
	i.e., $ \mathfrak{p} \notin \mathcal{B} $, and hence $ \mathfrak{p} \notin \mathcal{A} $, which is a contradiction. Therefore at least one of $ P_1 $ and $ P_2 $ must be a non-zero polynomial.
	
	Assume that $ P_1 $ is a non-zero polynomial. Then $ P_1 $ may have only finitely many roots. Therefore $ P_1(i) \neq 0 $ for all $ i \gg 0 $, which yields $ \operatorname{Tor}_{2i}^{A_{\mathfrak{p}}}(M_{\mathfrak{p}}, N_{\mathfrak{p}}) \neq 0 $ for all $ i \gg 0 $. So, in view of \eqref{theorem: asymptotic Ass on Tor: equation 2}, we obtain that
	\[
	\operatorname{Supp}\left( \operatorname{Tor}_i^{A_{\mathfrak{p}}}(M_{\mathfrak{p}}, N_{\mathfrak{p}}) \right) = \left\{ \mathfrak{p} A_{\mathfrak{p}} \right\} \quad \mbox{for all } i \gg 0,
	\]
	which implies that $ \mathfrak{p} \in \operatorname{Min}\left( \operatorname{Tor}_{2i}^A(M, N) \right) $ for all $ i \gg 0 $.
	
	Similarly, if $ P_2 $ is a non-zero polynomial, then we have that
	\[
	\mathfrak{p} \in \operatorname{Min}\left( \operatorname{Tor}_{2i+1}^A(M, N) \right) \quad\mbox{for all } i \gg 0.
	\]
	This completes the proof of the theorem.
\end{proof}
As a corollary of this theorem, we obtain the following result on associate primes.

\begin{corollary}\label{corollary: asymptotic Ass on Tor}
	Let $ A $ be a local complete intersection ring. Let $ M $ and $ N $ be finitely generated $ A $-modules. Then  exactly one of the following alternatives must hold:
	\begin{enumerate}[{\rm (1)}]
		\item $ \operatorname{Tor}_i^A(M, N) = 0 $ for all $ i > \dim(A) $;
		\item There exists a non-empty finite subset $ \mathcal{A} $ of $ \operatorname{Spec}(A) $ such that for every $ \mathfrak{p} \in \mathcal{A} $, at least one of the following statements holds true:
		\begin{enumerate}[{\rm (i)}]
			\item $ \mathfrak{p} \in \operatorname{Ass}_A\left( \operatorname{Tor}_{2i}^A(M, N) \right) $ for all $ i \gg 0 $;
			\item $ \mathfrak{p} \in \operatorname{Ass}_A\left( \operatorname{Tor}_{2i+1}^A(M, N) \right) $ for all $ i \gg 0 $.
		\end{enumerate}
	\end{enumerate}
\end{corollary}
\begin{proof}
	For every finitely generated $ A $-module $ D $, we have $ \operatorname{Min}_A(D) \subseteq \operatorname{Ass}_A(D) $. Therefore, if $ A $ is complete, then the corollary follows from Theorem~\ref{theorem: asymptotic Ass on Tor}. Now the general case can be deduced by using Lemma~\ref{lemma: Ass: Completion}.
\end{proof}

Here we give an example which shows that both statements (i) and (ii) in the assertion (2) of Corollary~\ref{corollary: asymptotic Ass on Tor} might not hold together.

\begin{example}\label{example: two sets of stable values: Ass: Tor}
	Let $ Q = k[[u,x]] $ be a ring of formal power series in variables $ u $ and $ x $ over a field $k$. We set $ A := Q/(ux) $ and $ M = N := Q/(u) $. Clearly, $ A $ is a local complete intersection ring, and $M$, $N$ are $A$-modules. Then, for every $ i \geqslant 1 $, we have that $ \operatorname{Tor}_{2i}^A(M, N) = 0 $ and $ \operatorname{Tor}_{2i-1}^A(M, N) \cong k $; see \cite[Example~4.3]{AB00}. So, for all $ i \geqslant 1 $, we obtain that
	\begin{align*}
	& \operatorname{Ass}_A\left( \operatorname{Tor}_{2i}^A(M, N) \right) = \phi \quad\mbox{and} \\
	& \operatorname{Ass}_A\left( \operatorname{Tor}_{2i-1}^A(M, N) \right) = \operatorname{Ass}_A(k) = \{ (u,x)/(ux) \}.
	\end{align*}
\end{example}

\subsection{Stability of primes in $ \operatorname{Ass}_A\left( \operatorname{Tor}_i^A(M,A/I^n) \right) $}

We now study the asymptotic stability of certain associated prime ideals of Tor-modules $ \operatorname{Tor}_i^A(M, A/I^n) $, $(i, n \geqslant 0)$, where $M$ is a finitely generated module over a local complete intersection ring $A$, and either $ I $ is a principal ideal or $ I $ has a principal reduction generated by an $ A $-regular element (see Corollary~\ref{corollary: asymptotic ass: Tor: for special ideals}). We start with the following lemma which we use in order to prove our result when $ I $ is a principal ideal.

\begin{lemma}\label{lemma: Tor: principal ideal}
	Let $ A $ be a ring, and $ M $ be an $ A $-module. Fix an element $ a \in A $. Then there exist an ideal $ J $ of $ A $ and a positive integer $ n_0 $ such that
	\[
	\operatorname{Tor}_{i + 2}^A\big(M, A/(a^n)\big) \cong \operatorname{Tor}_i^A(M, J) \quad \mbox{for all } i \geqslant 1 \mbox{ and } n \geqslant n_0.
	\]
\end{lemma}

\begin{proof}
	For every integer $ n \geqslant 1 $, we consider the following short exact sequence:
	\[
	0 \to (0 :_A a^n) \to A \to (a^n) \to 0,
	\]
	which yields the following isomorphisms:
	\begin{equation}\label{lemma: Tor: principal ideal: equation 1}
	\operatorname{Tor}_{i + 1}^A\big(M, (a^n)\big) \cong \operatorname{Tor}_i^A\big(M, (0 :_A a^n)\big) \quad \mbox{for all } i \geqslant 1.
	\end{equation}
	For every $ n \geqslant 1 $, the short exact sequence $ 0 \to (a^n) \to A \to A/(a^n) \to 0 $ gives
	\begin{equation}\label{lemma: Tor: principal ideal: equation 2}
	\operatorname{Tor}_{i + 1}^A\big(M, A/(a^n)\big) \cong \operatorname{Tor}_i^A\big(M, (a^n)\big) \quad \mbox{for all } i \geqslant 1.
	\end{equation}
	Thus \eqref{lemma: Tor: principal ideal: equation 1} and \eqref{lemma: Tor: principal ideal: equation 2} together yield
	\begin{equation}\label{lemma: Tor: principal ideal: equation 3}
	\operatorname{Tor}_{i + 2}^A\big(M, A/(a^n)\big) \cong \operatorname{Tor}_i^A\big(M, (0 :_A a^n)\big) \quad \mbox{for all } i \geqslant 1.
	\end{equation}
	Since $ A $ is a Noetherian ring, the ascending chain of ideals
	\[
	\big(0 :_A a\big) \subseteq \big(0 :_A a^2\big) \subseteq \big(0 :_A a^3\big) \subseteq \cdots
	\]
	will stabilize somewhere, i.e., there exists a positive integer $ n_0 $ such that
	\begin{equation}\label{lemma: Tor: principal ideal: equation 4}
	\big(0 :_A a^n\big) = \big(0 :_A a^{n_0}\big) \quad \mbox{for all } n \geqslant n_0.
	\end{equation}
	Then the lemma follows from \eqref{lemma: Tor: principal ideal: equation 3} and \eqref{lemma: Tor: principal ideal: equation 4} by setting $ J := \left(0 :_A a^{n_0}\right) $.
\end{proof}

Here is another lemma which we use in order to prove our result when $ I $ has a principal reduction generated by an $ A $-regular element.

\begin{lemma}\label{lemma: Tor: I has a principal reduction ideal}
	Let $ A $ be a ring. Let $ I $ be an ideal of $ A $ having a principal reduction generated by an $ A $-regular element. Then there exist an ideal $ J $ of $ A $ and a positive integer $ n_0 $ such that
	\[
	\operatorname{Tor}_i^A\big(M, A/I^n\big) \cong \operatorname{Tor}_i^A(M, A/J) \quad \mbox{for all } i \geqslant 2 \mbox{ and } n \geqslant n_0.
	\]
\end{lemma}

\begin{proof}
	Since $ I $ has a principal reduction generated by an $ A $-regular element, there exist an $ A $-regular element $ y $ and a positive integer $ n_0 $ such that
	\[
	I^{n+1} = y I^n ~\mbox{ for all } n \geqslant n_0.
	\]
	Then it can be shown that for every $ n \geqslant n_0 $, we obtain a short exact sequence:
	\begin{equation}\label{lemma: Tor: I has a principal reduction ideal: equation 1}
	0 \longrightarrow A/I^n \stackrel{y\cdot}{\longrightarrow} A/I^{n+1} \longrightarrow A/(y) \longrightarrow 0.
	\end{equation}
	Now note that $ \operatorname{Tor}_i^A\big(M, A/(y)\big) = 0 $ for all $ i \geqslant 2 $. Therefore the short exact sequence \eqref{lemma: Tor: I has a principal reduction ideal: equation 1} yields
	\[
	\operatorname{Tor}_i^A\big(M, A/I^n\big) \cong \operatorname{Tor}_i^A\big(M, A/I^{n+1}\big) \quad \mbox{for all } i \geqslant 2 \mbox{ and } n \geqslant n_0.
	\]
	Hence the lemma follows by setting $ J := I^{n_0} $.
\end{proof}

Now we can achieve one of the main goals of this article.

\begin{theorem}\label{theorem: asymptotic ass: Tor: for special ideals}
	Let $ A $ be as in {\rm Hypothesis~\ref{hypothesis 2}}. Let $ M $ be a finitely generated $ A $-module, and $ I $ be an ideal of $ A $. Suppose either $ I $ is principal or $ I $ has a principal reduction generated by an $ A $-regular element. Then there exist positive integers $ i_0 $ and $ n_0 $ such that exactly one of the following alternatives must hold:
	\begin{enumerate}[{\rm (1)}]
		\item $ \operatorname{Tor}_i^A(M, A/I^n) = 0 $ for all $ i \geqslant i_0 $ and $ n \geqslant n_0 $;
		\item There exists a non-empty finite subset $ \mathcal{A} $ of $ \operatorname{Spec}(A) $ such that for every $ \mathfrak{p} \in \mathcal{A} $, at least one of the following statements holds true:
		\begin{enumerate}[{\rm (i)}]
			\item $ \mathfrak{p} \in \operatorname{Min}\left( \operatorname{Tor}_{2i}^A(M, A/I^n) \right) $ for all $ i \geqslant i_0 $ and $ n \geqslant n_0 $;
			\item $ \mathfrak{p} \in \operatorname{Min}\left( \operatorname{Tor}_{2i+1}^A(M, A/I^n) \right) $ for all $ i \geqslant i_0 $ and $ n \geqslant n_0 $.
		\end{enumerate}
	\end{enumerate}
\end{theorem}

\begin{proof}
	If $ I $ is a principal ideal, then the result follows from Theorem~\ref{theorem: asymptotic Ass on Tor} and Lemma~\ref{lemma: Tor: principal ideal}. In another case, i.e., when $ I $ has a principal reduction generated by an $ A $-regular element, then we use Theorem~\ref{theorem: asymptotic Ass on Tor} and Lemma~\ref{lemma: Tor: I has a principal reduction ideal} to get the desired result of the theorem.
\end{proof}

As an immediate corollary of this theorem, we obtain the following:

\begin{corollary}\label{corollary: asymptotic ass: Tor: for special ideals}
	Let $ A $ be a local complete intersection ring. Let $ M $ be a finitely generated $ A $-module, and $ I $ be an ideal of $ A $. Suppose either $ I $ is principal or $ I $ has a principal reduction generated by an $ A $-regular element. Then there exist positive integers $ i_0 $ and $ n_0 $ such that exactly one of the following alternatives must hold:
	\begin{enumerate}[{\rm (1)}]
		\item $ \operatorname{Tor}_i^A(M, A/I^n) = 0 $ for all $ i \geqslant i_0 $ and $ n \geqslant n_0 $;
		\item There exists a non-empty finite subset $ \mathcal{A} $ of $ \operatorname{Spec}(A) $ such that for every $ \mathfrak{p} \in \mathcal{A} $, at least one of the following statements holds true:
		\begin{enumerate}[{\rm (i)}]
			\item $ \mathfrak{p} \in \operatorname{Ass}_A\left( \operatorname{Tor}_{2i}^A(M, A/I^n) \right) $ for all $ i \geqslant i_0 $ and $ n \geqslant n_0 $;
			\item $ \mathfrak{p} \in \operatorname{Ass}_A\left( \operatorname{Tor}_{2i+1}^A(M, A/I^n) \right) $ for all $ i \geqslant i_0 $ and $ n \geqslant n_0 $.
		\end{enumerate}
	\end{enumerate}
\end{corollary}

\begin{proof}
	Since $ \operatorname{Min}_A(D) \subseteq \operatorname{Ass}_A(D) $ for every finitely generated $ A $-module $ D $, the corollary follows from Theorem~\ref{theorem: asymptotic ass: Tor: for special ideals} when $ A $ is complete. 
	
	Now note that if $ I $ is principal, then so is its completion $ \widehat{I} $. Also note that if $ I $ has a principal reduction generated by an $ A $-regular element, then $ \widehat{I} $ has a principal reduction generated by an $ \widehat{A} $-regular element. It is well-known that
	\[
	\operatorname{Tor}_i^A(M, A/I^n) \otimes_A \widehat{A} ~\cong \operatorname{Tor}_i^{\widehat{A}} \left( \widehat{M}, \widehat{A}/{(\widehat{I})}^n \right) \quad \mbox{for all } i, n \geqslant 0.
	\]
	Therefore the general case can be easily deduced by using Lemma~\ref{lemma: Ass: Completion}.
\end{proof}

 \section*{Acknowledgements}

Ghosh was supported by DST, Govt.\,of India under the DST-INSPIRE Faculty Scheme. Mallick was supported by UGC, MHRD, Govt.\,of India.


\end{document}